\documentclass[9pt]{amsart}
\usepackage[utf8]{inputenc}
\usepackage{amscd}
\usepackage{amssymb}
\usepackage{pstricks}
\usepackage{empheq}
\usepackage{graphicx}
\usepackage{float}
\usepackage{todonotes}
\usepackage{subcaption}
\usepackage{xcolor}
\usepackage{textcomp}
\usepackage{cancel}
\usepackage[mathlines]{lineno}
\usepackage{comment}
\usepackage{ulem}
\newtheorem{theorem}{Theorem}

\newtheorem{lemma}{Lemma}
\newtheorem{definition}{Definition}

\newtheorem{remark}{Remark}
\newtheorem{example}{Example}
\numberwithin{equation}{section}
\numberwithin{lemma}{section}

\title[]{}

\title[]{$C^{1}$ Smoothness of the Conjugacy in a Delayed Nonautonomous Hartman--Grobman Theorem via $\mu$-Dichotomies
}
\author[Casta\~{n}eda]{\'Alvaro Casta\~{n}eda}
\author[Elorreaga]{Heli Elorreaga}
\address{Departamento de Matem\'aticas, Universidad de Chile, Casilla 653, Santiago, Chile.}
\email{castaneda@uchile.cl, heli.elorreaga@uchile.cl}
\thanks{This research has been supported by Agency for Research and Development, ANID-Chile through the grants ANID/FONDECYT 1240361 (\'A. Casta\~{n}eda) and the FONDECYT Postdoctorado project 3240354 (H. Elorreaga).}

\begin{document}

\begin{abstract}
We investigate the differentiability of the conjugacy in a nonautonomous version of the Hartman--Grobman Theorem for systems with finite delay, where the linear part satisfies a $\mu$-dichotomy. Under suitable conditions on the nonlinear perturbation, the conjugacy is shown to be a $C^{1}$ diffeomorphism.

\end{abstract}

\maketitle

\section{Introduction}
The notion of linearization near hyperbolic equilibria is essential for the qualitative analysis of dynamical systems.  The classical Hartman–Grobman theorem guarantees that nonlinear systems are locally topologically equivalent to their linearization in the vicinity of hyperbolic fixed points.  This fundamental outcome has been thoroughly generalized and enhanced, encompassing both nonautonomous and delayed equations.

 Palmer's \cite{Palmer} generalization in nonautonomous differential equations by using the concept of exponential dichotomy, establishing global topological equivalence between nonlinear systems and their linear part.  Subsequent study has concentrated on examining the differentiability characteristics of these conjugacies.  Castañeda and Robledo \cite{Castaneda4} delineated significant requirements for the establishment of differentiable conjugation, particularly emphasizing scenarios where the stable manifold is present solely. Later, Jara \cite{Jara} proves $C^2$ smoothness of nonautonomous linearization without spectral conditions, covering nonuniform hyperbolicities (including stable/unstable manifolds); Dragičević \textit{et al.} \cite{Dragicevic2} accomplished additional extensions to nonuniform exponential dichotomies and eliminated conventional resonance criteria, therefore greatly enhancing the application of differentiable linearization results.

 Concurrent developments in delay differential equations have also arisen.  Barreira and Valls \cite{B-v-per} established a Hartman–Grobman-type theorem for delayed nonautonomous systems, achieving conjugacies within the framework of unstable manifolds.  Expanding on this concept, G\'omez \textit{et al.} \cite{GCE} specifically examined the differentiability element, finding sufficient requirements for a $C^1$ diffeomorphism, contingent upon exponential decay rates and perturbation characteristics.

 In addition to these findings, there is a rapidly expanding and dynamic literature on general growth rates and $\mu$-dichotomies that extends the concept of nonuniform hyperbolicity beyond the exponential case. Silva introduced the nonuniform $\mu$-dichotomy spectrum and obtained reducibility/kinematic-similarity results unifying spectral descriptions for arbitrary growth rates \cite{Silva}. Backes and Dragičević characterized $(\mu,\nu)$-dichotomies via admissibility in weighted spaces, sharpening the structural role of growth rates and dispensing with bounded-growth/Lyapunov-norm assumptions \cite{BackesDragicevicAdmissibility}. Complementarily, Backes developed the $(\mu,\nu)$-dichotomy spectrum, describing its possible structures (in Banach settings with negative $\mu$-index of compactness) and giving normal-form applications \cite{BackesJMAA2025}. In parallel, Castañeda–Jara generalized Siegmund’s normal forms to systems with $\mu$-dichotomy, phrasing nonresonance conditions in terms of the associated $\mu$-spectrum \cite{CJ}. Very recently, Jara–Gallegos identified a spectrum invariance dilemma showing that nonuniform $\mu$-dichotomy spectra need not be preserved under nonuniform $(\mu,\varepsilon)$-kinematic similarity, and introduced quantitative devices (optimal ratio maps, $\varepsilon$-neighborhoods) to measure this loss \cite{GJ}.

 Inspired by these advances, this paper goes beyond prior research by treating delayed nonautonomous systems whose linear part exhibits nonuniform exponential or the more general $\mu$-dichotomy (as in Silva’s framework) \cite{Silva}. Under broad, verifiable hypotheses—without spectral gap or resonance constraints—we establish global differentiability of the topological conjugacy.
 
\subsection*{Structure of the article}

The structure of the article is as follows. In Section 2, we introduce the necessary background and define the functional setting, including the nonuniform dichotomy and growth conditions. In Section 3, we define the key operator $F$ whose fixed points yield the desired conjugacy. In last section, we state and prove the main result of the paper, showing the existence of a $C^1$-diffeomorphism that conjugates the linear and nonlinear systems. Finally, we provide an explicit example of parameters that satisfy all the hypotheses of the theorem, thus illustrating the applicability of our result.

\section{Basic Concepts ans Notations}
\subsection{Linear Delay Systems in $C([-r,0],\mathbb{R}^n)$ and Their Semigroups}

Given $r>0$, let $C:=C([-r,0],\mathbb{R}^n)$ be the Banach space of all continuous functions $\varphi:[-r,0]\to\mathbb{R}^n$ equipped with the supremum norm. We consider the following linear delay  equation
\begin{equation}\label{eq-lineal}
\begin{array}{rcl}
  x'(t) &=& L(t)x_t, \ \ t\geq s,\\
  x_s &=& \phi,
\end{array}
\end{equation}
where  $\phi, x_t\in C$ and  $x_t(\omega):=x(t+\omega)$, for  $\omega\in[-r,0]$.   $L(t):C\to \mathbb{R}^n$ are continuous linear operators for all $t\in\mathbb{R}$.

A continuous function $x:[s-r,a)\to \mathbb{R}^n$ where $a\leq+\infty$ is a solution of \eqref{eq-lineal} if $x$ is absolutely continuous on $[s,a)$  and satisfies \eqref{eq-lineal} for almost every $t\in[s,a)$. Note that for such $x$, it follows that $x_t\in C$ for every $t\in[s,a)$.  Standard results on the existence and uniqueness of solutions of a delay differential equation imply that, under the above assumptions on the operators $L(t)$, the equation \eqref{eq-lineal} has a unique solution on $[s-r,+\infty)$.   

Let  $x$ be the unique solution of (\ref{eq-lineal}). The solution operator $T(t,s):C\to C$ is defined by
\begin{equation*}
  T(t,s)\phi:=x_t, \qquad \mbox{ for } t\geq s \mbox{ and } \phi\in C,
\end{equation*}
 and it is a semigroup for $t\geq s$, linear and continuous in $\phi\in C$, and  strongly continuous in $t$ and $s$. That is, for $t$ fixed (respectively $s$ fixed) the map $s\mapsto T(t,s)\phi$ (respectively $t\mapsto T(t,s)\phi$) is continuous with the norm in $C$.

 \subsection{Growth Rates and $\mu$-Dichotomies}

 In order to capture, in a linear system, solution behaviors that go beyond the classical exponential patterns, Silva \cite{Silva} proposed the notion of a growth rate, which generalizes the exponential case. For completeness, we recall these definitions below.

 \begin{definition}
     A function $\mu:\mathbb{R}\to(0,+\infty)$ is said to be a growth rate if it is a strictly increasing function, $\mu(0)=1$, $\mu(t)\to+\infty$ as $t\to+\infty$ and $\mu(t)\to0$ as $t\to-\infty$. Moreover, if $\mu$ is differentiable, it is said a differentiable growth rate.
 \end{definition}
 
\begin{definition}\label{mu-D}
Let $\mu:\mathbb{R}\to (0,+\infty)$ be a growth rate. We will say that  \eqref{eq-lineal} has a nonuniform $\mu$-dichotomy on an  interval $J$, with constants $K\geq 1$, $\alpha, \beta>0$ and $\theta, \nu\geq 0$, satisfying $\alpha>\theta$ and $\beta>\nu,$  if there are  strongly continuous projections $P(s)$, $Q(s)=I-P(s)$, $s\in J$, such that:

 \begin{enumerate}
  \item[(i)] $T(t,s)P(s)=P(t)T(t,s)$, $t\geq s$ in $J$;
  \item[(ii)] $\overline{T}(t,s):=T(t,s)|_{\mathcal{R}Q(s)}$, $t\geq s$ is an isomorphism of $\mathcal{R}Q(s)$ onto $\mathcal{R}Q(t)$, and $\overline{T}(s,t):\mathcal{R}Q(t)\to\mathcal{R}Q(s)$ is defined as the inverse of $\overline{T}(t,s)$;
  \item[(iii)] $\|T(t,s)P(s)\|\leq K\left(\dfrac{\mu(t)}{\mu(s)}\right)^{-\alpha}\mu(s)^{\textnormal{sgn}(s)\theta}$, for  $t\geq s$ in $J$;
  \item[(iv)] $\|\overline{T}(t,s)Q(s)\|\leq K\left(\dfrac{\mu(t)}{\mu(s)}\right)^{\beta}\mu(s)^{\textnormal{sgn}(s)\nu}$, for $t\leq s$ in $J$.
\end{enumerate}
Moreover, if $\theta=\nu=0$, then we say that $T(t,s)$ admits uniform $\mu$-dichotomy on $J$.
\end{definition}

We define the stable and unstable subspaces of $T(t,s)$ at time $t$, respectively by
\begin{equation*}
E(t):=\mathcal{R}P(t) \ \ \mbox{ and } \ \ U(t):=\mathcal{R}Q(t), \ \ t\in J.
\end{equation*}

In this work, we focus on differentiable growth rates $\mu$ that satisfy the following key property:

\begin{enumerate}
    \item[(\textbf{H})] There exists a constant $N=N(r)>1$ such that
    \begin{equation*}
    \frac{\mu(s+r)}{\mu(s)}\leq N, \quad \textnormal{for all } s\in \mathbb{R}.
\end{equation*}
\end{enumerate}

As illustrative examples of these growth rates that generalize other dichotomies, we have:

\begin{example}
    Defining $\mu(t)=e^t$, for all $t\in\mathbb{R}$ and considering $\beta=\alpha$ and $\theta=\nu=0$, we recover the notion of exponential dichotomy. Moreover, if the parameters are considering as in Definition  \ref{mu-D}, we arrive at the concept of nonuniform exponential dichotomy. On the other hand, it can observe that this  growth rate is differentiable and satisfies \textbf{(H)}, where 
    $$\frac{\mu(s+r)}{\mu(s)}=e^r,$$
    and $N=e^r.$
\end{example}

\begin{example}
    Consider $\mu:\mathbb{R}\to(0,\infty)$ defined by
    $$\mu(t)=\left\{\begin{array}{cc}
        t+1, & t\geq0; \\
        \dfrac{1}{1-t}, & t\leq0.
    \end{array}\right.$$
    This function is a differentiable growth rate satisfying \textbf{(H)} with 
    $$\frac{\mu(s+r)}{\mu(s)}=\left\{\begin{array}{cc}
       \dfrac{s+r+1}{s+1},  & s\geq0; \\
        -x^2-rx+r+1, & -r\leq s\leq 0;\\
        \dfrac{1-s}{1-s-r}, & s\leq -r;
    \end{array}\right.$$
    and $N=\dfrac{r^2}{4}+r+1$.
\end{example}

\begin{example}
    The function $\mu:\mathbb{R}\to(0,+\infty)$ defined as
    $$\mu(t)=\left\{\begin{array}{cc}
        \ln(t+e), & t\geq0; \\
        \dfrac{1}{\ln(e-t)}, & t\leq0,
    \end{array}\right.$$
    is a differentiable growth rate satisfying \textbf{(H)} with 
    $$\frac{\mu(s+r)}{\mu(s)}=\left\{\begin{array}{cc}
       \dfrac{\ln(s+r+e)}{\ln(s+e)},  & s\geq0; \\
        \ln(s+r+e)\ln(e-s), & -r\leq s\leq 0;\\
        \dfrac{\ln(e-s)}{\ln(e-s-r)}, & s\leq -r;
    \end{array}\right.$$
    and $N=\ln^2(e+r/2)$.
    
\end{example}

\subsection{Perturbation}
We are interested in perturbations of the linear delay equation \eqref{eq-lineal} of the form

\begin{equation}\label{eq-per}
    x'(t) = L(t)x_t + g(t,x_t),
\end{equation}
where $g:\mathbb{R}\times C\to \mathbb{R}^n$ is a continuous function such that $g(t,0)=0$ for all $t\in\mathbb{R}$. We will also assume that $g$ is Lipschitz in the second variable. It can observe that, under these assumptions, results on the existence an uniqueness  of solutions of a delay equation guarantee that for a given $s\in\mathbb{R}$ and $\phi\in C$, the equation \eqref{eq-per} has a unique solution $x$ on $[s-r,+\infty)$ with $x_s=\phi.$ We refer to \cite{Hale} and \cite{Hale-Lunel} for more details.

Now, in order to describe the variation of constants formula, we must extend the linear operator $T(t,s)$ to the suitable phase space $C_0$, given by
$$C_0:=\{\phi:[-r,0]\to \mathbb{R}^n / \varphi \textnormal{ is continuous on } [-r,0), \exists \lim\limits_{\omega \to 0^{-}}\phi(\omega) \}.$$
This is a Banach space with the supremum norm. We know that, by the Riez's representation theorem, the linear operator $L(t):C\to\mathbb{R}^n$ can be written as a Riemann-Stieltjes integral:

\begin{equation}\label{Riez}
L(t)\phi = \int_{-r}^{0}d\eta(t,\omega)\phi(\omega),
\end{equation}
for some $n\times n$ matrix function $\eta(t,\omega)$ of bounded variation in $\omega\in[-r,0]$ for each $t\in\mathbb{R}$, measurable in $(t,\omega)\in\mathbb{R}\times\mathbb{R}$ and normalized in the following sense
$$\eta(t,\omega)=\left\{\begin{array}{rl}
    0, & \textnormal{ for } \omega\geq0,  \\
    \eta(t,-r), &  \textnormal{ for } \omega\leq -r,
\end{array}\right.$$
$\eta(t,\omega)$ is continuous from the left in $\theta\in(-r,0)$ and there exists some function $m\in L^1_{loc}(\mathbb{R})$ such that
$$Var_{[-r,0]}\eta(t,\cdot)\leq m(t), \quad \textnormal{ for all } t\in\mathbb{R}.$$
Since that each $\phi\in C_0$ is Riemann-Stieltjes integrable with respect to $\eta(t,\cdot)$ for each $t\in\mathbb{R}$, the linear operator $L(t)$ can be extended to $C_0$ using the integral in \eqref{Riez}. 

Next, for given $t,s\in\mathbb{R}$ with $t\geq s$, we can define a linear operator $T_0(t,s)$ on $C_0$ as
$$T_0(t,s)\phi=x_t \quad \textnormal{ for } t\geq s \textnormal{ and } \phi\in C_0,$$
where $x$ is the unique solution of equation \eqref{eq-lineal} on $[s-r,+\infty)$ with $x_s=\phi$. 

We assume throughout that the linear operator
 $T_0(t,s)$ has nonuniform $\mu$--bounded growth with parameter $\varepsilon>0$, that is, there are constants $\tilde{K}\geq1$, $a\geq0$ such that

\begin{equation}\label{BG}
    \|T_0(t,s)\|\leq \tilde{K}\left(\frac{\mu(t)}{\mu(s)}\right)^{\textnormal{sgn}(t-s)a}\mu(s)^{\textnormal{sgn}(s)\varepsilon}, \quad \textnormal{ for all } t,s\in\mathbb{R}.
\end{equation}

By the variation of constants formula for delay differential equations, we know that a function $x:[s-r,+\infty)\to\mathbb{R}^n$ is the unique solution of equation \eqref{eq-per} with $x_s=\phi$ if it satisfies 
\begin{equation*}
%\label{FVP}
    x_t = T(t,s)\phi + \int_s^t T_0(t,\tau)X_0g(\tau)d\tau,
\end{equation*}
for all $t\geq s$, where $X_0:\mathbb{R}^n\to C_0$ is a linear operator defined by
$$(X_0p)(\omega)=\left\{\begin{array}{rl}
    0, & \textnormal{ if } \omega\in[-r,0) \\
    p, &  \textnormal{ if } \omega=0,
\end{array}\right.$$
for each $p\in\mathbb{R}^n$.

On the other hand, assuming that the linear equation \eqref{eq-lineal} has nonuniform $\mu$-dichotomy on an  interval $J\supset \mathbb{R}_0^{+}$, for each $t\in J$ we define the linear operators $P_0(t), Q_0(t): \mathbb{R}^n\to C_0$ by
$$P_0(t)=X_0-Q_0(t) \quad \textnormal{ and } \quad Q_0(t)=\overline{T}(t,t+r)Q(t+r)T_0(t+r,t)X_0.$$

The following result extends to the $\mu$-dichotomy setting the bounds previously established in \cite[Proposition 4]{BV} and \cite[Proposition 1]{B-v-per}, now formulated in the space $C_0$ in terms of the operators $\overline{T}(t,s)$ and $T_0(t,s)$.

\begin{lemma}
%\label{Prop-ext}
    Consider a differentiable growth rate $\mu$ satisfying \textbf{\textbf{(H)}}. Assume that the linear equation \eqref{eq-lineal} has nonuniform $\mu$--dichotomy on some interval $J\supset \mathbb{R}_0^{+}$ and \eqref{BG} holds. If  \begin{subequations}\label{hyp:mu-dichotomy}
\begin{align}
\theta &\geq \varepsilon + \nu, \label{hyp:mu-dichotomy-theta} \\
\alpha &> \theta + \varepsilon, \label{hyp:mu-dichotomy-alpha} \\
\beta  &> \nu + \varepsilon. \label{hyp:mu-dichotomy-beta}
\end{align}
\end{subequations}

Then there exists a positive constant $D>0$ such that 
    \begin{equation}\label{B1}
        \|T_0(t,s)P_0(s)\|\leq D\left(\frac{\mu(t)}{\mu(s)}\right)^{-\alpha}\mu(s)^{\textnormal{sgn}(s)(\theta+\varepsilon)} , \quad \textnormal{ for } t\geq s \textnormal{ in } J;
    \end{equation}
    and
    \begin{equation}\label{B2}
         \|\overline{T}(t,s)Q_0(s)\|\leq D\left(\frac{\mu(t)}{\mu(s)}\right)^{\beta}\mu(s)^{\textnormal{sgn}(s)(\nu+\varepsilon)} , \quad \textnormal{ for } t\leq s \textnormal{ in } J.
    \end{equation}
\end{lemma}

\begin{proof}
We begin by proving the inequality \eqref{B2}.
    Take $p\in\mathbb{R}^n$. Then, from \textbf{(H)} and \eqref{BG}  we have that for $t\leq s$: 
    \begin{align*}
       \|\overline{T}(t,s)Q_0(s)p\| &= \|\overline{T}(t,s+r)Q(s+r)T_0(s+r,t)X_0 p\|\\
       &\leq K\left(\frac{\mu(t)}{\mu(s+r)}\right)^{\beta}\mu(s+r)^{\textnormal{sgn}(s+r)\nu}\tilde{K}\left(\frac{\mu(s+r)}{\mu(s)}\right)^{\textnormal{sgn}(r)a}\mu(s)^{\textnormal{sgn}(s)\varepsilon}\|p\|\\
       &= K\tilde{K}\left(\frac{\mu(t)}{\mu(s)}\right)^{\beta}\left(\frac{\mu(s+r)}{\mu(s)}\right)^{a-\beta}\mu(s+r)^{\textnormal{sgn}(s+r)\nu}\mu(s)^{\textnormal{sgn}(s)\varepsilon}\|p\|\\
       &\leq K\tilde{K}N^{|a-\beta|}\left(\frac{\mu(t)}{\mu(s)}\right)^{\beta}\mu(s+r)^{\textnormal{sgn}(s+r)\nu}\mu(s)^{\textnormal{sgn}(s)\varepsilon}\|p\|\\
       &=  K\tilde{K}N^{|a-\beta|}\left(\frac{\mu(t)}{\mu(s)}\right)^{\beta}\left(\frac{\mu(s+r)}{\mu(s)}\right)^{\textnormal{sgn}(s+r)\nu}\mu(s)^{\textnormal{sgn}(s)\varepsilon+\textnormal{sgn}(s+r)\nu}\|p\|\\
       &\leq K\tilde{K}N^{|a-\beta|+\nu}\left(\frac{\mu(t)}{\mu(s)}\right)^{\beta}\mu(s)^{\textnormal{sgn}(s)\varepsilon+\textnormal{sgn}(s+r)\nu}\|p\|.
    \end{align*}

    Observe that for $0\leq t\leq s$ we get that $\mu(s)\geq 1$ and so

    \begin{align*}
         \|\overline{T}(t,s)Q_0(s)p\| &\leq K\tilde{K}N^{|a-\beta|+\nu}\left(\frac{\mu(t)}{\mu(s)}\right)^{\beta}\mu(s)^{\varepsilon+\nu}\|p\|\\
         &=  K\tilde{K}N^{|a-\beta|+\nu}\left(\frac{\mu(t)}{\mu(s)}\right)^{\beta}\mu(s)^{\textnormal{sgn}(s)(\varepsilon+\nu)}\|p\|.
    \end{align*}
    The case $t\leq 0\leq s$ is similar to the previous one. For the case $t\leq s\leq 0$ we have that $\mu(s)\leq1$. Moreover,  we will consider two subcases: 
    \begin{enumerate}
        \item[(i)] When $-r\leq s\leq 0$: 
        \begin{align*}
            \|\overline{T}(t,s)Q_0(s)p\| &\leq K\tilde{K}N^{|a-\beta|+\nu}\left(\frac{\mu(t)}{\mu(s)}\right)^{\beta}\mu(s)^{-\varepsilon+\nu}\|p\|\\
            &\leq K\tilde{K}N^{|a-\beta|+\nu}\left(\frac{\mu(t)}{\mu(s)}\right)^{\beta}\mu(s)^{\textnormal{sgn}(s)(\varepsilon+\nu)}\|p\|.
        \end{align*}

        \item[(ii)] When $s\leq -r< 0$:
        \begin{align*}
           \|\overline{T}(t,s)Q_0(s)p\| &\leq K\tilde{K}N^{|a-\beta|+\nu}\left(\frac{\mu(t)}{\mu(s)}\right)^{\beta}\mu(s)^{-\varepsilon-\nu}\|p\|\\
           &= K\tilde{K}N^{|a-\beta|+\nu}\left(\frac{\mu(t)}{\mu(s)}\right)^{\beta}\mu(s)^{\textnormal{sgn}(s)(\nu+\varepsilon)}\|p\|.
        \end{align*}
    \end{enumerate}
    Hence,
    $$\|\overline{T}(t,s)Q_0(s)\| \leq K\tilde{K}N^{|a-\beta|+\nu}\left(\frac{\mu(t)}{\mu(s)}\right)^{\beta}\mu(s)^{\textnormal{sgn}(s)(\nu+\varepsilon)}, \quad \textnormal{ for } t\leq s.$$
   This proves the inequality \eqref{B2}. Combined with condition \eqref{hyp:mu-dichotomy-beta}, it yields the expected relation between the nonuniformity bound and the decay rate.

On the other hand, taking $t=s$ in the above inequality we obtain
    $$\|Q_0(s)\|\leq K\tilde{K}N^{|a-\beta|+\nu}\mu(s)^{\textnormal{sgn}(s)(\nu+\varepsilon)},$$
    and thus, for $t\in[s,s+r]$, from \eqref{BG} we have
    \begin{align*}
        \|T_0(t,s)P_0(s)\| &\leq \tilde{K}\left(\frac{\mu(t)}{\mu(s)}\right)^{a}\mu(s)^{\textnormal{sgn}(s)\varepsilon}\|X_0-Q_0(s)\|\\
        &\leq \tilde{K}\left(\frac{\mu(s+r)}{\mu(s)}\right)^{a}\mu(s)^{\textnormal{sgn}(s)\varepsilon}(1+K\tilde{K}N^{|a-\beta|+\nu}\mu(s)^{\textnormal{sgn}(s)(\nu+\varepsilon)})\\
        &\leq \tilde{K}N^{a}\mu(s)^{\textnormal{sgn}(s)\varepsilon}\left(1+K_1\mu(s)^{\textnormal{sgn}(s)(\nu+\varepsilon)}\right),
    \end{align*}

    where $K_1=K\tilde{K}N^{|a-\beta|+\nu}$.
    
    Having established the inequality \eqref{B2}, we now turn to the proof of the inequality \eqref{B1}. For $s\geq0$, by \eqref{hyp:mu-dichotomy-theta} and $\mu(s)\geq1$, the last inequality implies that
    \begin{align*}
        \|T_0(t,s)P_0(s)\| &\leq \tilde{K}N^{a}\mu(s)^{\varepsilon}(1+K_1\mu(s)^{\nu+\varepsilon})\\
        &\leq \tilde{K}N^{a}\mu(s)^{\varepsilon}(\mu(s)^{\theta}+K_1\mu(s)^{\theta}),
    \end{align*}
    and so
    \begin{equation}\label{P1-eq1}
        \|T_0(t,s)P_0(s)\| \leq \tilde{K}N^{a}(1+K_1)\mu(s)^{\textnormal{sgn}(s)(\theta+\varepsilon)}.
    \end{equation}
    
    For the case when $s\leq 0$, \eqref{hyp:mu-dichotomy-theta} and $\mu(s)\leq1$ imply that

    \begin{align*}
        \|T_0(t,s)P_0(s)\| &\leq \tilde{K}N^{a}\mu(s)^{-\varepsilon}(1+K_1\mu(s)^{-\nu-\varepsilon})\\
        &\leq \tilde{K}N^{a}\mu(s)^{-\varepsilon}(\mu(s)^{-\theta}+K_1\mu(s)^{-\theta}),
    \end{align*}
    and so
    \begin{equation}\label{P1-eq2}
        \|T_0(t,s)P_0(s)\| \leq \tilde{K}N^{a}(1+K_1)\mu(s)^{\textnormal{sgn}(s)(\theta+\varepsilon)}.
    \end{equation}
    Hence, for $t\in[s,s+r]$ from \eqref{P1-eq1} and \eqref{P1-eq2}  it follows that
    \begin{equation}\label{P1-eq3}
        \|T_0(t,s)P_0(s)\| \leq \tilde{K}N^{a}(1+K_1)\mu(s)^{\textnormal{sgn}(s)(\theta+\varepsilon)}.
    \end{equation}

    Now, taking $t\geq s+r$ we have that
    \begin{align*}
        T_0(t,s)P_0(s) &= T_0(t,s)(X_0-Q_0(s))\\
        &= T(t,s+r)T_0(s+r,s)(X_0-\overline{T}(s,s+r)Q(s+r)T_0(s+r,s)X_0)\\
        &= T(t,s+r)(T_0(s+r,s)X_0-Q(s+r)T_0(s+r,s)X_0)\\
        &=T(t,s+r)P(s+r)T_0(s+r,s)X_0,
    \end{align*}
    and then, the inequality (iii) in Definition \ref{mu-D} and \eqref{BG} imply that
    \begin{align*}
        \|T_0(t,s)P_0(s)\| &\leq K\left(\frac{\mu(t)}{\mu(s+r)}\right)^{-\alpha}\mu(s+r)^{\textnormal{sgn}(s+r)\theta}\tilde{K}\left(\frac{\mu(s+r)}{\mu(s)}\right)^{a}\mu(s)^{\textnormal{sgn}(s)\varepsilon}\\
        &\leq K\tilde{K}N^{a+\alpha}\left(\frac{\mu(t)}{\mu(s)}\right)^{-\alpha}\left(\frac{\mu(s+r)}{\mu(s)}\right)^{\textnormal{sgn}(s+r)\theta}\mu(s)^{\textnormal{sgn}(s)\varepsilon+\textnormal{sgn}(s+r)\nu}\\
        &\leq K\tilde{K}N^{a+\alpha+\theta}\left(\frac{\mu(t)}{\mu(s)}\right)^{-\alpha}\mu(s)^{\textnormal{sgn}(s)\varepsilon+\textnormal{sgn}(s+r)\nu}.
    \end{align*}

    As before, we will analyze three cases for the last inequality. First, if $s\geq0$ it can be seen immediately that

    \begin{equation}\label{P1-eq4}
        \|T_0(t,s)P_0(s)\| \leq K\tilde{K}N^{a+\alpha+\theta}\left(\frac{\mu(t)}{\mu(s)}\right)^{-\alpha}\mu(s)^{\textnormal{sgn}(s)(\theta+\varepsilon)}.
    \end{equation}

    Next, if $-r\leq s\leq0$ we have
    \begin{align*}
       \|T_0(t,s)P_0(s)\| &\leq  K\tilde{K}N^{a+\alpha+\theta}\left(\frac{\mu(t)}{\mu(s)}\right)^{-\alpha} \mu(s)^{\theta-\varepsilon}\\
       &\leq K\tilde{K}N^{a+\alpha+\theta}\left(\frac{\mu(t)}{\mu(s)}\right)^{-\alpha} \mu(s)^{-\theta-\varepsilon},
    \end{align*}
    which implies that
    \begin{equation}\label{P1-eq5}
        \|T_0(t,s)P_0(s)\|\leq K\tilde{K}N^{a+\alpha+\theta}\left(\frac{\mu(t)}{\mu(s)}\right)^{-\alpha}\mu(s)^{\textnormal{sgn}(s)(\theta+\varepsilon)}.
    \end{equation}

    When $s\leq-r<0$, proceeding as in the previous cases, it is possible to prove that
    \begin{equation}\label{P1-eq6}
        \|T_0(t,s)P_0(s)\|\leq K\tilde{K}N^{a+\alpha+\varepsilon}\left(\frac{\mu(t)}{\mu(s)}\right)^{-\alpha}\mu(s)^{\textnormal{sgn}(s)(\theta+\varepsilon)}.
    \end{equation}
    Therefore, from \eqref{P1-eq4}, \eqref{P1-eq5} and \eqref{P1-eq6} we obtain that for $t\geq s+r$:

    \begin{equation}\label{P1-eq7}
        \|T_0(t,s)P_0(s)\|\leq K\tilde{K}N^{a+\alpha+\theta}\left(\frac{\mu(t)}{\mu(s)}\right)^{-\alpha}\mu(s)^{\textnormal{sgn}(s)(\theta+\varepsilon)}.
    \end{equation}

    Finally, the inequality \eqref{B1} follows from \eqref{P1-eq3} and \eqref{P1-eq7}. Combined with condition \eqref{hyp:mu-dichotomy-alpha}, it yields the expected relation between the nonuniformity bound and the growth rate, thereby completing the proof.

\end{proof}

 \begin{remark}
We emphasize that the condition \eqref{hyp:mu-dichotomy-theta} appears explicitly at various points in the proof. The other two inequalities, \eqref{hyp:mu-dichotomy-alpha} and \eqref{hyp:mu-dichotomy-beta}, while not used directly, play a structural role in ensuring the compatibility between the nonuniformity bounds and the decay/growth rates specified in Definition \ref{mu-D}.
\end{remark}

\section{Topological equivalence}

In this section we establish a partial version of the Hartman-Grobman theorem for a type of perturbations of a nonuniform $\mu$--dichotomy. For this, we will consider the following assumptions:

\begin{enumerate}
    \item[(H1)] The equation \eqref{eq-lineal} has nonuniform $\mu$--dichotomy on $\mathbb{R}$ with a differentiable growth rate $\mu$ satisfying \textbf{(H)}, the linear operator
 $T_0(t,s)$ satisfies  \eqref{BG}  and  the parameters $\alpha$, $\beta$, $\theta$, $\nu$ and $\varepsilon$ fulfill \eqref{hyp:mu-dichotomy}.

    \item[(H2)] There exist constants $\delta>0$ and $\gamma>\max\{\theta,\nu\}$ such that
    \begin{equation}\label{g-cond}
        |g(t,\phi)-g(t,\psi)|\leq \delta\min\{1,\|\phi-\psi\|\}\mu'(t)\mu(t)^{-\textnormal{sgn}(t)(\gamma+\varepsilon)-1},
    \end{equation}
    for every $t\in\mathbb{R}$ and $\phi, \psi\in C$.
\end{enumerate}
Inspired by the approach developed in \cite{B-v-per}, we denote by $\mathcal{M}$ the set of all bounded continuous functions
\begin{equation*}
%\label{eta}
   \eta\colon\{(t,b) : t\in\mathbb{R}, b\in U(t)\}\to C, 
\end{equation*}
equipped with the norm
\begin{equation}\label{eta-norm}
    \|\eta\|_{\infty}:=\sup\{\|\eta(t,b)\| : t\in\mathbb{R}, b\in U(t)\}<+\infty,
\end{equation}
where $U(t)$ is the unstable subspace of $T(t,s)$ at time $t$. It can easily shown that $\mathcal{M}$ is a Banach space with the norm $\|\cdot\|_{\infty}$ defined in  \eqref{eta-norm}. We also write
$$\eta^t=\eta(t,\cdot) \quad \textnormal{ and } \quad h^t=Id_{U(t)} + \eta^t,$$
and as in \cite{B-v-per}, the solutions $x$ of \eqref{eq-per} in the evolutionary form:
\begin{equation*}
    x_t=R(t,s)(x_s) \quad \textnormal{ for } t\geq s.
\end{equation*}

\begin{theorem}\label{Th-Equiv}
    Assume that the conditions (H1) and (H2) hold. If $\delta$ is sufficiently small, then there exists a unique $\eta\in\mathcal{M}$ such that $h^{t}=Id_{U(t)}+\eta^t$ and 
    \begin{equation}\label{conj}
        h^t\circ T(t,s) = R(t,s)\circ h^s \quad \textnormal{ on } U(s),
    \end{equation}
    for every $t, s\in\mathbb{R}$ with $t\geq s$. Moreover, each map $h^t$ is a homeomorphism.
\end{theorem}

\begin{proof}
The proof proceeds in three steps. First, we show that the operator $F$ is well defined. Second, we prove that $F$ is a contraction. Finally, we verify that each map $h^t$ is a homeomorphism.

\textit{Step 1. $F$ is well defined.} We define the operator $F\colon\mathcal{M}\to \mathcal{M}$ by
    \begin{equation}\label{F}
        \begin{array}{rcl}
            F(\eta)(t,b) & = & \displaystyle\int_{-\infty}^{t} T_0(t,\tau)P_0(\tau)g(\tau,\overline{T}(\tau,t)b+\eta(\tau,\overline{T}(\tau,t)b))d\tau \\
             & & -\displaystyle\int_{t}^{+\infty}\overline{T}(t,\tau)Q_0(\tau)g(\tau, T(\tau,t)b+\eta(\tau,T(\tau,t)b))d\tau
        \end{array}
    \end{equation}
    for $\eta\in\mathcal{M}$, $t\in\mathbb{R}$ and $b\in U(t)$.

    Since $g(t,0)=0$ for all $t\in \mathbb{R}$, it follows from \eqref{B1}, \eqref{B2} and \eqref{g-cond} that
\begin{align*}
        \int_{-\infty}^{t}\|T_0(t,\tau)P_0(\tau)\|\cdot|g(\tau,\overline{T}(\tau,t)b&+\eta(\tau,\overline{T}(\tau,t)b))|d\tau\\
        &\leq D\delta\mu(t)^{-\alpha}\int_{-\infty}^{t} \mu(\tau)^{\alpha-1}\mu'(\tau)\mu(\tau)^{\textnormal{sgn}(\tau)(\theta-\gamma)}d\tau\\
        &\leq \frac{D\delta}{\alpha},
    \end{align*}
    and
\begin{align*}
\int_{t}^{+\infty}\|\overline{T}(t,\tau)Q_0(\tau)\|\cdot|g(\tau,T(\tau,t)b&+\eta(\tau,T(\tau,t)b))|d\tau\\
        &\leq D\delta\mu(t)^{\beta}\int_{t}^{+\infty} \mu(\tau)^{-\beta-1}\mu'(\tau)\mu(\tau)^{\textnormal{sgn}(\tau)(\nu-\gamma)}d\tau\\
        &\leq \frac{D\delta}{\beta}.
    \end{align*}
    Hence,
    \begin{equation*}
    %\label{F-cota}
\|F(\eta)\|_{\infty}\leq D\delta\left(\frac{\alpha+\beta}{\alpha\beta}\right),
    \end{equation*}
    which shows that the operator $F$ is well defined. 

\textit{Step 2. $F$ is a contraction.} In order to prove that there exists a unique $\eta\in\mathcal{M}$ such that $F(\eta)=\eta$ we will apply Banach's fixed point theorem. For this, observe that we just need to show that the operator $F$ is a contraction. In fact, for each $\eta, \xi\in\mathcal{M}$ we have
\begin{equation*}
        \begin{split}
        F(\eta)(t,b) &- F(\xi)(t,b)\\
        =& \int_{-\infty}^{t}T_0(t,\tau)P_0(\tau)[g(\tau,\overline{T}(\tau,t)b+\eta(\tau,\overline{T}(\tau,t)b))\\
        &- g(\tau,\overline{T}(\tau,t)b+\xi(\tau,\overline{T}(\tau,t)b))]d\tau\\
        &- \int_{t}^{+\infty}\overline{T}(t,\tau)Q_0(\tau)[g(\tau,T(\tau,t)b+\eta(\tau,T(\tau,t)b))\\
        &-g(\tau,T(\tau,t)b+\xi(\tau,T(\tau,t)b))d\tau.
    \end{split}
    \end{equation*}

    From \eqref{B1}, \eqref{B2} and \eqref{g-cond} we obtain that
\begin{equation*}
    \begin{split}
        \int_{-\infty}^{t}&\|T_0(t,\tau)P_0(\tau)\|\cdot|g(\tau,\overline{T}(\tau,t)b+\eta(\tau,\overline{T}(\tau,t)b))- g(\tau,\overline{T}(\tau,t)b+\xi(\tau,\overline{T}(\tau,t)b))|d\tau\\
        \leq & D\delta\mu(t)^{-\alpha}\int_{-\infty}^{t}\mu(\tau)^{\alpha-1}\mu'(\tau)\|\eta(\tau,\overline{T}(\tau,t)b) - \xi(\tau,\overline{T}(\tau,t)b)\|\mu(\tau)^{\textnormal{sgn}(\tau)(\theta-\gamma)}d\tau\\
        \leq & D\delta\|\eta-\xi\|_{\infty}\mu(t)^{-\alpha}\int_{-\infty}^{t}\mu(\tau)^{\alpha-1}\mu'(\tau)d\tau\\
        =& \frac{D\delta}{\alpha}\|\eta-\xi\|_{\infty},
    \end{split}
\end{equation*}

and

\begin{equation*}
    \begin{split}
        \int_{t}^{+\infty}&\overline{T}(t,\tau)Q_0(\tau)[g(\tau,T(\tau,t)b+\eta(\tau,T(\tau,t)b))
        -g(\tau,T(\tau,t)b+\xi(\tau,T(\tau,t)b))d\tau\\
        \leq & D\delta\mu(t)^{\beta}\int_{t}^{+\infty}\mu(\tau)^{-\beta-1}\mu'(\tau)\|\eta(\tau,T(\tau,t)b) - \xi(\tau,T(\tau,t)b)\|\mu(\tau)^{\textnormal{sgn}(\tau)(\nu-\gamma)}d\tau\\
        \leq & D\delta\|\eta-\xi\|_{\infty}\mu(t)^{\beta}\int_{t}^{+\infty}\mu(\tau)^{-\beta-1}\mu'(\tau)d\tau\\
        =& \frac{D\delta}{\beta}\|\eta-\xi\|_{\infty}.
    \end{split}
\end{equation*}

Thus,
\begin{equation*}
%\label{Fcont1}
    \|F(\eta)-F(\xi)\|_{\infty}\leq D\delta\left(\frac{\alpha+\beta}{\alpha\beta}\right)\|\eta-\xi\|_{\infty}.
\end{equation*}

From the above inequality we can conclude that for $\delta$ sufficiently small the operator $F$ is a contraction and then there exists a unique $\eta\in\mathcal{M}$ satisfying $F(\eta)=\eta$.

On the other hand, in \cite[Lemma 2]{B-v-per} was proved that the property \eqref{conj} is holds if and only if $F(\eta)=\eta.$  
It thus remains to carry out the third step.

\textit{Step 3. $h^t$ is a homeomorphism for each $t \in \mathbb{R}.$} Indeed, since that $\|\eta\|_{\infty}=\|F(\eta)\|_{\infty}<1$  and $h^t=Id_{U(t)}+\eta^{t}$ we can deduce that each function $h^t$ is a homeomorphism. This conclude the proof of the theorem.

\end{proof}

\section{Differentiability}
%\label{Diff}

In this section, we investigate the differentiability properties of the conjugacy established in Theorem~\ref{Th-Equiv}. As in the previous section, we work under assumption~(H1). In addition, we will now impose the following further assumptions:

\begin{enumerate}
    \item[(H3)] $g$ is a $C^1$--function with $D_2g(t,0)=0$.
    \item[(H4)] Let \(\delta > 0\) and \(\gamma > \max\{\theta, \nu\}\) be as in assumption~\textnormal{(H2)}. Then, there exist constants \(\lambda > 0\) and \(\xi\) satisfying
\[
\max\{\nu + \varepsilon,\; |a - \beta| + \varepsilon\} < \xi < \dfrac{\alpha + \beta}{2},
\]
such that
\begin{equation}\label{g-cond2}
    |g(t,\phi) - g(t,\psi)| \leq \delta \min\{1,\|\phi - \psi\|_{\mu}\} \, \mu'(t)\mu(t)^{-\textnormal{sgn}(t)(\gamma + \varepsilon) - 1}
\end{equation}
and
\begin{equation}\label{Dg-cond}
    |D_2g(t,\phi) - D_2g(t,\psi)| \leq \lambda \min\{1,\|\phi - \psi\|_{\mu}\} \, \mu'(t)\mu(t)^{-\textnormal{sgn}(t)(\gamma + \varepsilon) - 2\textnormal{sgn}(t)\xi - 1},
\end{equation}
for every \(t \in \mathbb{R}\) and \(\phi, \psi \in C\), where \(\|\cdot\|_{\mu} := \|\cdot\|\mu(t)^{-\textnormal{sgn}(t)(\xi + \varepsilon)}\).

\end{enumerate}

In order to formulate our main result in this section, we denote by $\mathcal{M}^1$ the set of all functions
\begin{equation*}
%\label{eta}
   \eta\colon\{(t,b) : t\in\mathbb{R}, b\in U(t)\}\to C, 
\end{equation*}
that are strongly continuous in $t$, $C^1$ in $b$, equipped with the norm
\begin{equation}\label{eta-norm1}
    \|\eta\|_{1,\mu} = \|\eta\|_{\infty,\mu} + \left\|\frac{\partial \eta}{\partial b}\right\|_{\infty,\mu}<+\infty,
\end{equation}
where
\begin{equation*}
    \|\eta\|_{\infty,\mu}:=\sup\{\|\eta(t,b)\|_{\mu} : t\in\mathbb{R}, b\in U(t)\}.
\end{equation*}
 $\mathcal{M}^1$ is a Banach space with the norm $\|\cdot\|_{1,\mu}$ defined in  \eqref{eta-norm1}. Moreover, for $q>0$, we define  the following closed neighborhood

 \begin{equation*}
     \mathcal{B}=\overline{B}(0,q):=\{\eta\in\mathcal{M}^1 : \|\eta\|_{1,\mu}\leq q\}.
 \end{equation*}

Before stating the main result, we present two preparatory lemmas that are essential for its proof. The first lemma shows that the operator \(F\), defined in \eqref{F} and now considered on the space \(\mathcal{B}\), is well defined. The second establishes that \(F\) is a contraction on \(\mathcal{B}\).

\begin{lemma}\label{well defined}
     Assume that conditions \textnormal{(H1)}, \textnormal{(H3)}, and \textnormal{(H4)} hold, and let \(q > 0\). If the parameters \(\delta\) and \(\lambda\) satisfy
\begin{align}
\delta &\leq \frac{q\alpha\beta}{D(\alpha+\beta)(1+q)^2}, \label{ineq-delta} \\
\lambda &\leq \frac{q^2(\alpha+\beta - 2\xi)\left[(\xi - \varepsilon)^2 - (a - \beta)^2\right]}{D\left[D\left((\xi - \varepsilon)^2 - (a - \beta)^2\right) + 4\xi\tilde{K}(\alpha + \beta - 2\xi)\right](1 + q)^3}, \label{ineq-lambda}
\end{align}
where \(D\) is the constant appearing in inequalities \eqref{B1} and \eqref{B2}, then the operator \(F \colon \mathcal{B} \to \mathcal{B}\), as given in \eqref{F}, is well defined.
\end{lemma}

\begin{proof}
    We define the operator $F\colon\mathcal{B}\to \mathcal{B}$ by
    \begin{equation*}
    %\label{F2}
        \begin{array}{rcl}
            F(\eta)(t,b) & = & \displaystyle\int_{-\infty}^{t} T_0(t,\tau)P_0(\tau)g(\tau,\overline{T}(\tau,t)b+\eta(\tau,\overline{T}(\tau,t)b))d\tau \\
             & & -\displaystyle\int_{t}^{+\infty}\overline{T}(t,\tau)Q_0(\tau)g(\tau, T(\tau,t)b+\eta(\tau,T(\tau,t)b))d\tau
        \end{array}
    \end{equation*}
    for $\eta\in\mathcal{B}$, $t\in\mathbb{R}$ and $b\in U(t)$.

    From \eqref{g-cond2} considering $|g(t,\phi)|\leq \delta\mu'(t)\mu(t)^{-\textnormal{sgn}(t)(\gamma+\varepsilon)-1}$  and Step 1 in the proof of Theorem \ref{Th-Equiv}, it follows that
    $$\|F(\eta)(t,b)\|\mu(t)^{-\textnormal{sgn}(t)(\xi+\varepsilon)}\leq D\delta\left(\frac{\alpha+\beta}{\alpha\beta}\right)\mu(t)^{-\textnormal{sgn}(t)(\xi+\varepsilon)}\leq D\delta\left(\frac{\alpha+\beta}{\alpha\beta}\right),$$
    and so
    \begin{equation}\label{F-cota-mu}
        \|F(\eta)\|_{\infty,\mu} \leq D\delta\left(\frac{\alpha+\beta}{\alpha\beta}\right).
    \end{equation}

    On the other hand, we have 
$$ \displaystyle{\frac{\partial F(\eta)}{\partial b}(t,b)}= I_1 + I_2 + I_3 + I_4,$$
where
$$I_1 = \displaystyle\int_{-\infty}^{t}T_0(t,\tau)P_0(\tau)D_2g(\tau,\overline{T}(\tau,t)b+\eta(\tau,\overline{T}(\tau,t)b))\overline{T}(\tau,t)d\tau;$$
$$I_2 =  \displaystyle\int_{-\infty}^{t}T_0(t,\tau)P_0(\tau)D_2g(\tau,\overline{T}(\tau,t)b+\eta(\tau,\overline{T}(\tau,t)b))\frac{\partial\eta}{\partial b}(\tau,\overline{T}(\tau,t)b)\overline{T}(\tau,t)d\tau ;$$
$$I_3 = -\displaystyle\int_{t}^{+\infty}\overline{T}(t,\tau)Q_0(\tau)D_2g(\tau,T(\tau,t)b+\eta(\tau,T(\tau,t)b))T(\tau,t)d\tau $$
and
$$I_4 = -\displaystyle\int_{t}^{+\infty}\overline{T}(t,\tau)Q_0(\tau)D_2g(\tau,T(\tau,t)b+\eta(\tau,T(\tau,t)b))\frac{\partial\eta}{\partial b}(\tau,T(\tau,t)b)T(\tau,t)d\tau.$$
    
We analyze each term individually, providing separate estimates for the cases \(t \geq 0\) and \(t < 0\). These bounds collectively guarantee that the operator \(F\) is well defined on the entire real line.

Considering $t\geq 0$, from \eqref{BG}, \eqref{B1}, \eqref{B2} and \eqref{Dg-cond} together
with the fact that $D_2g(t,0)=0$ and $b\in U(t)$, we obtain for $I_1:$
\begin{equation*}
    \begin{split}
        \mu&(t)^{-\textnormal{sgn}(t)(\xi+\varepsilon)}\int_{-\infty}^{t}\|T_0(t,\tau)P_0(\tau)\| |D_2g(\tau, \overline{T}(\tau,t)b+\eta(\tau,\overline{T}(\tau,t)b))| \|\overline{T}(\tau,t)\|d\tau\\
        &\leq D^2\lambda\mu(t)^{-\alpha-\beta-\textnormal{sgn}(t)(\xi-\nu)}\int_{-\infty}^{t}\mu(\tau)^{\alpha+\beta-2\textnormal{sgn}(\tau)(\xi)-1}\mu'(\tau)\mu(\tau)^{-\textnormal{sgn}(\tau)(\gamma-\theta)}d\tau\\
        &\leq D^2\lambda\mu(t)^{-\alpha-\beta-(\xi-\nu)}\int_{-\infty}^{t}\mu(\tau)^{\alpha+\beta-2\textnormal{sgn}(\tau)(\xi)-1}\mu'(\tau)d\tau\\
        &= D^2\lambda\mu(t)^{-\alpha-\beta-(\xi-\nu)}\Big[\int_{-\infty}^{0}\mu(\tau)^{(\alpha+\beta+2\xi)-1}\mu'(\tau)d\tau+\int_{0}^{t}\mu(\tau)^{(\alpha+\beta-2\xi)-1}\mu'(\tau)d\tau\Big]\\
        &= D^2\lambda\mu(t)^{-\alpha-\beta-(\xi-\nu)}\Big[\Big(\frac{1}{\alpha+\beta+2\xi}-\frac{1}{\alpha+\beta-2\xi}\Big)+\frac{\mu(t)^{\alpha+\beta-2\xi}}{\alpha+\beta-2\xi}\Big]\\
        &< \frac{D^2\lambda}{\alpha+\beta-2\xi}\mu(t)^{-3\xi+\nu}\\
        &\leq \frac{D^2\lambda}{\alpha+\beta-2\xi}.
    \end{split}
    \end{equation*}
    Let us now consider the term \(I_2\). Assuming \(t \geq 0\), we obtain:
\begin{equation*}
    \begin{split}
        \mu&(t)^{-\textnormal{sgn}(t)(\xi+\varepsilon)}\int_{-\infty}^{t}\|T_0(t,\tau)P_0(\tau)\| |D_2g(\tau, \overline{T}(\tau,t)b+\eta(\tau,\overline{T}(\tau,t)b))| \Big\|\frac{\partial\eta}{\partial b}(\tau,\overline{T}(\tau,t)b)\Big\|\|\overline{T}(\tau,t)\|d\tau\\
        &\leq D^2\lambda\mu(t)^{-\alpha-\beta-\textnormal{sgn}(t)(\xi-\nu)}\int_{-\infty}^{t}\mu(\tau)^{\alpha+\beta-2\textnormal{sgn}(\tau)(\xi)-1}\mu'(\tau)\mu(\tau)^{-\textnormal{sgn}(\tau)(\gamma-\theta)}\Big\|\frac{\partial\eta}{\partial b}\Big\|_{\infty,\mu}\mu(\tau)^{\textnormal{sgn}(\tau)(\xi+\varepsilon)}d\tau\\
        &\leq  D^2\lambda\Big\|\frac{\partial\eta}{\partial b}\Big\|_{\infty,\mu}\mu(t)^{-\alpha-\beta-(\xi-\nu)}\int_{-\infty}^{t}\mu(\tau)^{\alpha+\beta-\textnormal{sgn}(\tau)(\xi)+\textnormal{sgn}(\tau)(\varepsilon)-1}\mu'(\tau)d\tau\\
        &\leq D^2\lambda q \mu(t)^{-\alpha-\beta-(\xi-\nu)} \Big[\int_{-\infty}^{0}\mu(\tau)^{(\alpha+\beta+\xi-\varepsilon)-1}\mu'(\tau)d\tau + \int_{0}^{t}\mu(\tau)^{(\alpha+\beta-\xi+\varepsilon)-1}\mu'(\tau)d\tau\Big]\\
        &= D^2\lambda q \mu(t)^{-\alpha-\beta-(\xi-\nu)}\Big[\Big(\frac{1}{\alpha+\beta+\xi-\varepsilon}-\frac{1}{\alpha+\beta-\xi+\varepsilon}\Big)+\frac{\mu(t)^{\alpha+\beta-\xi+\varepsilon}}{\alpha+\beta-\xi+\varepsilon}\Big]\\
        &< \frac{D^2\lambda q}{\alpha+\beta-\xi+\varepsilon}\mu(t)^{-2\xi+\nu+\varepsilon}\\
        &\leq \frac{D^2\lambda q}{\alpha+\beta-\xi+\varepsilon}.
    \end{split}
\end{equation*}

The estimation of \(I_3\) follows a similar strategy. For \(t \geq 0\), we have:
\begin{equation*}
    \begin{split}
        \mu&(t)^{-\textnormal{sgn}(t)(\xi+\varepsilon)}\int_{t}^{+\infty}\|\overline{T}(t,\tau)Q_0(\tau)\| |D_2g(\tau,T(\tau,t)b+\eta(\tau,T(\tau,t)b))| \|T(\tau,t)\|d\tau\\
        &\leq D\lambda\tilde{K}\mu(t)^{\beta-a-\textnormal{sgn}(t)(\xi)}\int_{t}^{+\infty}\mu(\tau)^{a-\beta-2\textnormal{sgn}(\tau)(\xi)-1}\mu'(\tau)\mu(\tau)^{-\textnormal{sgn}(\tau)(\gamma-\nu)}d\tau\\
        &\leq D\lambda\tilde{K}\mu(t)^{\beta-a-\xi}\int_{t}^{+\infty}\mu(\tau)^{-(2\xi-a+\beta)-1}\mu'(\tau)d\tau\\
        &= \frac{D\lambda\tilde{K}}{2\xi-a+\beta}\mu(t)^{-3\xi}\\
        &\leq \frac{D\lambda\tilde{K}}{2\xi-a+\beta};
    \end{split}
\end{equation*}
and finally, we estimate \(I_4\). From the assumptions and previously derived bounds, it follows that:
\begin{equation*}
    \begin{split}
        \mu&(t)^{-\textnormal{sgn}(t)(\xi+\varepsilon)}\int_{t}^{+\infty}\|\overline{T}(t,\tau)Q_0(\tau)\| |D_2g(\tau,T(\tau,t)b+\eta(\tau,T(\tau,t)b))| \Big\|\frac{\partial\eta}{\partial b}(\tau,T(\tau,t)b)\Big\| \|T(\tau,t)\|d\tau\\
        &\leq D\lambda\tilde{K}\mu(t)^{\beta-a-\textnormal{sgn}(t)(\xi)}\int_{t}^{+\infty}\mu(\tau)^{a-\beta-2\textnormal{sgn}(\tau)(\xi)-1}\mu'(\tau)\mu(\tau)^{-\textnormal{sgn}(\tau)(\gamma-\nu)}\Big\|\frac{\partial\eta}{\partial b}\Big\|_{\infty,\mu}\mu(\tau)^{\textnormal{sgn}(\tau)(\xi+\varepsilon)}d\tau\\
        &\leq D\lambda\tilde{K}\Big\|\frac{\partial\eta}{\partial b}\Big\|_{\infty,\mu}\mu(t)^{\beta-a-\xi}\int_{t}^{+\infty}\mu(\tau)^{a-\beta-\textnormal{sgn}(\tau)(\xi)+\textnormal{sgn}(\tau)(\varepsilon)-1}\mu'(\tau)d\tau\\
        &\leq D\lambda\tilde{K} q\mu(t)^{\beta-a-\xi}\int_{t}^{+\infty}\mu(\tau)^{-(\xi-a+\beta-\varepsilon)-1}\mu'(\tau)d\tau\\
        &= \frac{D\lambda\tilde{K}q}{\xi-a+\beta-\varepsilon}\mu(t)^{-2\xi+\varepsilon}\\
        &\leq \frac{D\lambda\tilde{K}q}{\xi-a+\beta-\varepsilon}.
    \end{split}
\end{equation*}

Summarizing, for $t\geq0$, the above inequalities imply that
\begin{align}
    \left\|\frac{\partial F(\eta)}{\partial b}(t,b)\right\|_{\infty,\mu} &< \frac{D^2\lambda}{\alpha+\beta-2\xi} + \frac{D^2\lambda q}{\alpha+\beta-\xi+\varepsilon} + \frac{D\lambda\tilde{K}}{2\xi-a+\beta} + \frac{D\lambda\tilde{K}q}{\xi-a+\beta-\varepsilon}\nonumber\\
    &< \frac{D^2\lambda}{\alpha+\beta-2\xi}(1+q) + \frac{D\lambda\tilde{K}}{\xi-a+\beta-\varepsilon}(1+q).\label{Dftmayorcero}
\end{align}

Now, we consider the case when $t<0$. Since \eqref{BG}, \eqref{B1}, \eqref{B2} and \eqref{Dg-cond} together
with the fact that $D_2g(t,0)=0$ and $b\in U(t)$, we get for $I_1:$

\begin{equation*}
    \begin{split}
        \mu&(t)^{-\textnormal{sgn}(t)(\xi+\varepsilon)}\int_{-\infty}^{t}\|T_0(t,\tau)P_0(\tau)\| |D_2g(\tau, \overline{T}(\tau,t)b+\eta(\tau,\overline{T}(\tau,t)b))| \|\overline{T}(\tau,t)\|d\tau\\
        &\leq D^2\lambda\mu(t)^{-\alpha-\beta-\textnormal{sgn}(t)(\xi-\nu)}\int_{-\infty}^{t}\mu(\tau)^{\alpha+\beta-2\textnormal{sgn}(\tau)(\xi)-1}\mu'(\tau)\mu(\tau)^{-\textnormal{sgn}(\tau)(\gamma-\theta)}d\tau\\
        &\leq D^2\lambda\mu(t)^{-\alpha-\beta+\xi-\nu}\int_{-\infty}^{t}\mu(\tau)^{(\alpha+\beta+2\xi)-1}\mu'(\tau)d\tau\\
        &= \frac{D^2\lambda}{\alpha+\beta+2\xi}\mu(t)^{3\xi-\nu}\\
        &\leq \frac{D^2\lambda}{\alpha+\beta+2\xi}.
    \end{split}
\end{equation*}

We proceed with the term \(I_2\) in the case \(t < 0\). As in the previous estimates, we obtain:
\begin{equation*}
    \begin{split}
        \mu&(t)^{-\textnormal{sgn}(t)(\xi+\varepsilon)}\int_{-\infty}^{t}\|T_0(t,\tau)P_0(\tau)\| |D_2g(\tau, \overline{T}(\tau,t)b+\eta(\tau,\overline{T}(\tau,t)b))| \Big\|\frac{\partial\eta}{\partial b}(\tau,\overline{T}(\tau,t)b)\Big\|\|\overline{T}(\tau,t)\|d\tau\\
        &\leq D^2\lambda\mu(t)^{-\alpha-\beta-\textnormal{sgn}(t)(\xi-\nu)}\int_{-\infty}^{t}\mu(\tau)^{\alpha+\beta-2\textnormal{sgn}(\tau)(\xi)-1}\mu'(\tau)\mu(\tau)^{-\textnormal{sgn}(\tau)(\gamma-\theta)}\Big\|\frac{\partial\eta}{\partial b}\Big\|_{\infty,\mu}\mu(\tau)^{\textnormal{sgn}(\tau)(\xi+\varepsilon)}d\tau\\
        &\leq  D^2\lambda\Big\|\frac{\partial\eta}{\partial b}\Big\|_{\infty,\mu}\mu(t)^{-\alpha-\beta+\xi-\nu}\int_{-\infty}^{t}\mu(\tau)^{\alpha+\beta-\textnormal{sgn}(\tau)(\xi)+\textnormal{sgn}(\tau)(\varepsilon)-1}\mu'(\tau)d\tau\\
        &\leq D^2\lambda q \mu(t)^{-\alpha-\beta+\xi-\nu} \int_{-\infty}^{t}\mu(\tau)^{(\alpha+\beta+\xi-\varepsilon)-1}\mu'(\tau)d\tau\\
        &= \frac{D^2\lambda q}{\alpha+\beta+\xi-\varepsilon}\mu(t)^{2\xi-\nu-\varepsilon}\\
        &\leq \frac{D^2\lambda q}{\alpha+\beta+\xi-\varepsilon}.
    \end{split}
\end{equation*}

For \(t < 0\), similar reasoning yields the following bound for \(I_3\):
\begin{equation*}
    \begin{split}
        \mu&(t)^{-\textnormal{sgn}(t)(\xi+\varepsilon)}\int_{t}^{+\infty}\|\overline{T}(t,\tau)Q_0(\tau)\| |D_2g(\tau,T(\tau,t)b+\eta(\tau,T(\tau,t)b))| \|T(\tau,t)\|d\tau\\
        &\leq D\lambda\tilde{K}\mu(t)^{\beta-a-\textnormal{sgn}(t)(\xi)}\int_{t}^{+\infty}\mu(\tau)^{a-\beta-2\textnormal{sgn}(\tau)(\xi)-1}\mu'(\tau)\mu(\tau)^{-\textnormal{sgn}(\tau)(\gamma-\nu)}d\tau\\
        &\leq D\lambda\tilde{K}\mu(t)^{\beta-a+\xi}\int_{t}^{+\infty}\mu(\tau)^{a-\beta-2\textnormal{sgn}(\tau)(\xi)-1}\mu'(\tau)d\tau\\
        &= D\lambda\tilde{K}\mu(t)^{\beta-a+\xi}\Big[\int_{t}^{0}\mu(\tau)^{(a-\beta+2\xi)-1}\mu'(\tau)d\tau + \int_{0}^{+\infty}\mu(\tau)^{(a-\beta-2\xi)-1}\mu'(\tau)d\tau\Big]\\
        &= D\lambda\tilde{K}\mu(t)^{\beta-a+\xi}\Big[\frac{4\xi}{4\xi^2-(a-\beta)^2}-\frac{\mu(t)^{a-\beta+2\xi}}{a-\beta+2\xi}\Big]\\
        &< \frac{4\xi D\lambda\tilde{K}}{4\xi^2-(a-\beta)^2}\mu(t)^{\beta-a+\xi}\\
        &\leq \frac{4\xi D\lambda\tilde{K}}{4\xi^2-(a-\beta)^2},
    \end{split}
\end{equation*}

and finally, for $I_4$ we have:

\begin{equation*}
    \begin{split}
        \mu&(t)^{-\textnormal{sgn}(t)(\xi+\varepsilon)}\int_{t}^{+\infty}\|\overline{T}(t,\tau)Q_0(\tau)\| |D_2g(\tau,T(\tau,t)b+\eta(\tau,T(\tau,t)b))| \Big\|\frac{\partial\eta}{\partial b}(\tau,T(\tau,t)b)\Big\| \|T(\tau,t)\|d\tau\\
        &\leq D\lambda\tilde{K}\mu(t)^{\beta-a-\textnormal{sgn}(t)(\xi)}\int_{t}^{+\infty}\mu(\tau)^{a-\beta-2\textnormal{sgn}(\tau)(\xi)-1}\mu'(\tau)\mu(\tau)^{-\textnormal{sgn}(\tau)(\gamma-\nu)}\Big\|\frac{\partial\eta}{\partial b}\Big\|_{\infty,\mu}\mu(\tau)^{\textnormal{sgn}(\tau)(\xi+\varepsilon)}d\tau\\
        &\leq D\lambda\tilde{K}\Big\|\frac{\partial\eta}{\partial b}\Big\|_{\infty,\mu}\mu(t)^{\beta-a+\xi}\int_{t}^{+\infty}\mu(\tau)^{a-\beta-\textnormal{sgn}(\tau)(\xi)+\textnormal{sgn}(\tau)(\varepsilon)-1}\mu'(\tau)d\tau\\
        &\leq D\lambda\tilde{K} q\mu(t)^{\beta-a+\xi}\Big[\int_{t}^{0}\mu(\tau)^{(a-\beta+\xi-\varepsilon)-1}\mu'(\tau)d\tau + \int_{0}^{+\infty}\mu(\tau)^{-(\xi-a+\beta-\varepsilon)-1}\mu'(\tau)d\tau\Big]\\
        &= D\lambda\tilde{K} q\mu(t)^{\beta-a+\xi}\Big[\frac{2(\xi-\varepsilon)}{(\xi-\varepsilon)^2-(a-\beta)^2}-\frac{\mu(t)^{a-\beta+\xi-\varepsilon}}{a-\beta+\xi-\varepsilon}\Big]\\
        &< \frac{2(\xi-\varepsilon)D\lambda\tilde{K}q}{(\xi-\varepsilon)^2-(a-\beta)^2}\mu(t)^{\beta-a+\xi}\\
        &\leq \frac{2(\xi-\varepsilon)D\lambda\tilde{K}q}{(\xi-\varepsilon)^2-(a-\beta)^2}.
    \end{split}
\end{equation*}

Collecting the above estimates, we conclude that for \(t < 0\),

\begin{align}
    \left\|\frac{\partial F(\eta)}{\partial b}(t,b)\right\|_{\infty,\mu} &< \frac{D^2\lambda}{\alpha+\beta+2\xi} + \frac{D^2\lambda q}{\alpha+\beta+\xi-\varepsilon} + \frac{4\xi D\lambda\tilde{K}}{4\xi^2-(a-\beta)^2} + \frac{2(\xi-\varepsilon)D\lambda\tilde{K}q}{(\xi-\varepsilon)^2-(a-\beta)^2}\nonumber\\
    &< \frac{D^2\lambda}{\alpha+\beta+\xi-\varepsilon}(1+q) + \frac{4\xi D\lambda\tilde{K}}{(\xi-\varepsilon)^2-(a-\beta)^2}(1+q).\label{DF-tmenorcero}
\end{align}

Therefore, from the inequalities \eqref{Dftmayorcero} and \eqref{DF-tmenorcero}, it follows that for $t\in\mathbb{R}$:
\begin{align}
    \left\|\frac{\partial F(\eta)}{\partial b}(t,b)\right\|_{\infty,\mu}&<\Big[\frac{D^2\lambda}{\alpha+\beta-2\xi}+\frac{4\xi D\lambda\tilde{K}}{(\xi-\varepsilon)^2-(a-\beta)^2}\Big](1+q)\nonumber\\
    &= \frac{D\lambda[D((\xi-\varepsilon)^2-(a-\beta)^2)+4\xi\tilde{K}(\alpha+\beta-2\xi)]}{(\alpha+\beta-2\xi)[(\xi-\varepsilon)^2-(a-\beta)^2]}(1+q).\label{DF-cota-mu}
\end{align}

From \eqref{F-cota-mu} and \eqref{DF-cota-mu} we obtain that
\begin{align*}
    \|F(\eta)\|_{1,\mu} &= \|F(\eta)\|_{\infty,\mu} +  \left\|\frac{\partial F(\eta)}{\partial b}\right\|_{\infty,\mu}\nonumber\\
    &< D\delta\left(\frac{\alpha+\beta}{\alpha\beta}\right) + \frac{D\lambda[D((\xi-\varepsilon)^2-(a-\beta)^2)+4\xi\tilde{K}(\alpha+\beta-2\xi)]}{(\alpha+\beta-2\xi)[(\xi-\varepsilon)^2-(a-\beta)^2]}(1+q).
\end{align*}
Combining the above inequality with \eqref{ineq-delta} and \eqref{ineq-lambda}, we obtain

$$\|F(\eta)\|_{1,\mu}<\frac{q}{(1+q)^2}+\frac{q^2}{(1+q)^2}=\frac{q}{1+q}<q,$$
and thus the result follows.
\end{proof}

    \begin{lemma}\label{lemma:contraction}
Assume that conditions \textnormal{(H1)}, \textnormal{(H3)} and \textnormal{(H4)} hold, and let \(q > 0\). If the parameters \(\delta\) and \(\lambda\) satisfy inequalities \eqref{ineq-delta} and \eqref{ineq-lambda}, respectively, then the operator \(F \colon \mathcal{B} \to \mathcal{B}\), defined in \eqref{F}, is a contraction.
\end{lemma}

\begin{proof}
For each $\eta_1, \eta_2\in\mathcal{B}$ we have

    \begin{equation*}
        \begin{split}
        F(\eta_1)(t,b) &- F(\eta_2)(t,b)\\
        =& \int_{-\infty}^{t}T_0(t,\tau)P_0(\tau)[g(\tau,\overline{T}(\tau,t)b+\eta_1(\tau,\overline{T}(\tau,t)b))\\
        &- g(\tau,\overline{T}(\tau,t)b+\eta_2(\tau,\overline{T}(\tau,t)b))]d\tau\\
        &- \int_{t}^{+\infty}\overline{T}(t,\tau)Q_0(\tau)[g(\tau,T(\tau,t)b+\eta_1(\tau,T(\tau,t)b))\\
        &-g(\tau,T(\tau,t)b+\eta_2(\tau,T(\tau,t)b))d\tau.
    \end{split}
    \end{equation*}

    From \eqref{B1}, \eqref{B2} and \eqref{g-cond2} we obtain that
\begin{equation*}
    \begin{split}
        \int_{-\infty}^{t}&\|T_0(t,\tau)P_0(\tau)\|\cdot|g(\tau,\overline{T}(\tau,t)b+\eta_1(\tau,\overline{T}(\tau,t)b))- g(\tau,\overline{T}(\tau,t)b+\eta_2(\tau,\overline{T}(\tau,t)b))|d\tau\\
        \leq & D\delta\mu(t)^{-\alpha}\int_{-\infty}^{t}\mu(\tau)^{\alpha-1}\mu'(\tau)\|\eta_1(\tau,\overline{T}(\tau,t)b) - \eta_2(\tau,\overline{T}(\tau,t)b)\|\mu(\tau)^{-\textnormal{sgn}(\tau)(\xi+\varepsilon)}d\tau\\
        \leq & D\delta\|\eta_1-\eta_2\|_{\infty,\mu}\mu(t)^{-\alpha}\int_{-\infty}^{t}\mu(\tau)^{\alpha-1}\mu'(\tau)d\tau\\
        =& \frac{D\delta}{\alpha}\|\eta_1-\eta_2\|_{\infty,\mu},
    \end{split}
\end{equation*}

and

\begin{equation*}
    \begin{split}
        \int_{t}^{+\infty}&\overline{T}(t,\tau)Q_0(\tau)[g(\tau,T(\tau,t)b+\eta_1(\tau,T(\tau,t)b))
        -g(\tau,T(\tau,t)b+\eta_2(\tau,T(\tau,t)b))d\tau\\
        \leq & D\delta\mu(t)^{\beta}\int_{t}^{+\infty}\mu(\tau)^{-\beta-1}\mu'(\tau)\|\eta_1(\tau,T(\tau,t)b) - \eta_2(\tau,T(\tau,t)b)\|\mu(\tau)^{-\textnormal{sgn}(\tau)(\xi+\varepsilon)}d\tau\\
        \leq & D\delta\|\eta_1-\eta_2\|_{\infty,\mu}\mu(t)^{\beta}\int_{t}^{+\infty}\mu(\tau)^{-\beta-1}\mu'(\tau)d\tau\\
        =& \frac{D\delta}{\beta}\|\eta_1-\eta_2\|_{\infty,\mu}.
    \end{split}
\end{equation*}

Thus,
\begin{equation}\label{Fcont}
    \|F(\eta_1)-F(\eta_2)\|_{\infty,\mu}\leq D\delta\left(\frac{\alpha+\beta}{\alpha\beta}\right)\|\eta_1-\eta_2\|_{\infty,\mu}.
\end{equation}
Now, we  study the difference
\[
\frac{\partial F(\eta_1)}{\partial b}(t,b) - \frac{\partial F(\eta_2)}{\partial b}(t,b),
\]
which can be decomposed into the sum of four terms:
\[
J_1(t) + J_2(t) + J_3(t) + J_4(t),
\]
where 

\begin{equation*}
\renewcommand{\arraystretch}{1.6}
\begin{aligned}
J_1 &= \int_{-\infty}^{t} T_0(t,\tau) P_0(\tau) 
\big[ D_2g(\tau, \overline{T}(\tau,t)b + z_1(\tau)) 
     - D_2g(\tau, \overline{T}(\tau,t)b + z_2(\tau)) \big] 
\, \overline{T}(\tau,t) \, d\tau, \\
J_2 &= \int_{-\infty}^{t} T_0(t,\tau) P_0(\tau) 
\Big[ D_2g(\tau, \overline{T}(\tau,t)b + z_1(\tau)) w_1(\tau) 
     - D_2g(\tau, \overline{T}(\tau,t)b + z_2(\tau)) w_2(\tau) \Big] 
\, \overline{T}(\tau,t) \, d\tau, \\
J_3 &= -\int_{t}^{+\infty} \overline{T}(t,\tau) Q_0(\tau) 
\big[ D_2g(\tau, T(\tau,t)b + \tilde{z}_1(\tau)) 
     - D_2g(\tau, T(\tau,t)b + \tilde{z}_2(\tau)) \big] 
\, T(\tau,t) \, d\tau, \\
J_4 &= -\int_{t}^{+\infty} \overline{T}(t,\tau) Q_0(\tau) 
\Big[ D_2g(\tau, T(\tau,t)b + \tilde{z}_1(\tau)) \tilde{w}_1(\tau) 
     - D_2g(\tau, T(\tau,t)b + \tilde{z}_2(\tau)) \tilde{w}_2(\tau) \Big] 
\, T(\tau,t) \, d\tau,
\end{aligned}
\end{equation*}
with
\begin{equation*}
%\label{eq:J-decomposition}
%\renewcommand{\arraystretch}{1.2}
\begin{aligned}
 &z_i(\tau) := \eta_i(\tau, \overline{T}(\tau,t)b), 
&& w_i(\tau) := \dfrac{\partial \eta_i}{\partial b}(\tau, \overline{T}(\tau,t)b),\\
            &\tilde{z}_i(\tau) := \eta_i(\tau, T(\tau,t)b), 
&& \tilde{w}_i(\tau) := \dfrac{\partial \eta_i}{\partial b}(\tau, T(\tau,t)b).
\end{aligned}
\end{equation*}
We begin by estimating each term \(J_i(t)\), \(i=1,\dots,4\), in the case \(t \geq 0\). Using the bounds from \eqref{BG}, \eqref{B1}, \eqref{B2}, and the Lipschitz-type condition \eqref{Dg-cond}, along with the growth assumptions on \(\mu\), we obtain upper bounds for each term; for $I_1$ we get:
\begin{align*}
&\mu(t)^{-\textnormal{sgn}(t)(\xi+\varepsilon)} 
\int_{-\infty}^{t} 
  \|T_0(t,\tau) P_0(\tau)\| 
  \left| D_2g(\tau, \overline{T}(\tau,t)b + z_1(\tau)) 
       - D_2g(\tau, \overline{T}(\tau,t)b + z_2(\tau)) \right| 
  \|\overline{T}(\tau,t)\| \, d\tau \\
&\leq D^2 \lambda \, \mu(t)^{-\alpha - \beta - \text{sgn}(t)(\xi - \nu)} 
     \int_{-\infty}^{t} 
     \mu(\tau)^{\alpha + \beta - 2\text{sgn}(\tau)\xi - 1} \mu'(\tau) \,
     \|z_1(\tau) - z_2(\tau)\| 
     \mu(\tau)^{-\text{sgn}(\tau)(\xi+\varepsilon)} \, d\tau \\
&\leq D^2 \lambda \, \|\eta_1 - \eta_2\|_{\infty,\mu} 
     \, \mu(t)^{-\alpha - \beta - (\xi - \nu)} 
     \left[
       \int_{-\infty}^{0} \mu(\tau)^{\alpha + \beta + 2\xi - 1} \mu'(\tau) \, d\tau
       + \int_{0}^{t} \mu(\tau)^{\alpha + \beta - 2\xi - 1} \mu'(\tau) \, d\tau
     \right] \\
&= D^2 \lambda \, \|\eta_1 - \eta_2\|_{\infty,\mu} 
     \, \mu(t)^{-\alpha - \beta - (\xi - \nu)} 
     \left[
       \frac{1}{\alpha + \beta + 2\xi} 
       - \frac{1}{\alpha + \beta - 2\xi} 
       + \frac{\mu(t)^{\alpha + \beta - 2\xi}}{\alpha + \beta - 2\xi}
     \right] \\
&< \frac{D^2 \lambda}{\alpha + \beta - 2\xi} 
    \, \mu(t)^{-3\xi + \nu} \, \|\eta_1 - \eta_2\|_{\infty,\mu} \\
&\leq \frac{D^2 \lambda}{\alpha + \beta - 2\xi} \, \|\eta_1 - \eta_2\|_{\infty,\mu}.
\end{align*}
We now turn to the estimate of \(J_2(t)\):
\small{\begin{align*}
&\mu(t)^{-\textnormal{sgn}(t)(\xi+\varepsilon)} 
\int_{-\infty}^{t} 
  \|T_0(t,\tau)P_0(\tau)\| 
  \left| 
    D_2g(\tau, \overline{T}(\tau,t)b + z_1(\tau)) w_1(\tau) 
    - D_2g(\tau, \overline{T}(\tau,t)b + z_2(\tau)) w_2(\tau) 
  \right| 
  \|\overline{T}(\tau,t)\| d\tau \\
&\leq \mu(t)^{-\textnormal{sgn}(t)(\xi+\varepsilon)} 
\int_{-\infty}^{t} 
  D \left( \frac{\mu(t)}{\mu(\tau)} \right)^{-\alpha} 
  \Big[ 
    |D_2g(\tau, \overline{T}(\tau,t)b + z_1(\tau))| \cdot \|w_1(\tau) - w_2(\tau)\| \\
&\quad + 
    |D_2g(\tau, \overline{T}(\tau,t)b + z_1(\tau)) 
    - D_2g(\tau, \overline{T}(\tau,t)b + z_2(\tau))| \cdot \|w_2(\tau)\| 
  \Big] 
  \mu(\tau)^{\text{sgn}(\tau)(\xi+\varepsilon)} 
  D \left( \frac{\mu(\tau)}{\mu(t)} \right)^{\beta} 
  \mu(t)^{\text{sgn}(t)(\nu+\varepsilon)} d\tau \\
&\leq D^2 \lambda \mu(t)^{-\alpha - \beta - \text{sgn}(t)(\xi - \nu)} 
\int_{-\infty}^{t} 
  \mu(\tau)^{\alpha + \beta - \text{sgn}(\tau)(\xi - \varepsilon) - 1} \mu'(\tau) 
  \left( 
    \|w_1 - w_2\|_{\infty,\mu} + q \|\eta_1 - \eta_2\|_{\infty,\mu} 
  \right) d\tau \\
&= D^2 \lambda \left( 
    \|w_1 - w_2\|_{\infty,\mu} + q \|\eta_1 - \eta_2\|_{\infty,\mu} 
  \right) \mu(t)^{-\alpha - \beta - (\xi - \nu)} 
\left[
  \int_{-\infty}^{0} \mu(\tau)^{\alpha + \beta + \xi - \varepsilon - 1} \mu'(\tau) d\tau \right. \\
&\qquad\left. + \int_0^t \mu(\tau)^{\alpha + \beta - \xi + \varepsilon - 1} \mu'(\tau) d\tau
\right] \\
&= D^2 \lambda \left( 
    \|w_1 - w_2\|_{\infty,\mu} + q \|\eta_1 - \eta_2\|_{\infty,\mu} 
  \right) \mu(t)^{-\alpha - \beta - (\xi - \nu)} 
\left[
  \frac{1}{\alpha + \beta + \xi - \varepsilon} 
  - \frac{1}{\alpha + \beta - \xi + \varepsilon} 
  + \frac{\mu(t)^{\alpha + \beta - \xi + \varepsilon}}{\alpha + \beta - \xi + \varepsilon} 
\right] \\
&< \frac{D^2 \lambda}{\alpha + \beta - \xi + \varepsilon} 
   \mu(t)^{-2\xi + \nu + \varepsilon} 
   \left( 
     \|w_1 - w_2\|_{\infty,\mu} + q \|\eta_1 - \eta_2\|_{\infty,\mu} 
   \right) \\
&\leq \frac{D^2 \lambda}{\alpha + \beta - \xi + \varepsilon} 
      \|w_1 - w_2\|_{\infty,\mu} 
      + \frac{D^2 \lambda q}{\alpha + \beta - \xi + \varepsilon} 
      \|\eta_1 - \eta_2\|_{\infty,\mu}.
\end{align*}}
We now proceed to estimate the third term, \(J_3(t)\):
\begin{align*}
&\mu(t)^{-\textnormal{sgn}(t)(\xi + \varepsilon)} 
\int_t^{+\infty} 
  \|\overline{T}(t,\tau) Q_0(\tau)\| 
  \left| D_2g(\tau, T(\tau,t)b + \tilde{z}_1(\tau)) 
       - D_2g(\tau, T(\tau,t)b + \tilde{z}_2(\tau)) \right| 
  \|T(\tau,t)\| \, d\tau \\
&\leq D \lambda \tilde{K} \, \mu(t)^{\beta - a - \textnormal{sgn}(t)\xi} 
\int_t^{+\infty} 
  \mu(\tau)^{a - \beta - 2\textnormal{sgn}(\tau)\xi - 1} \mu'(\tau) \,
  \|\tilde{z}_1(\tau) - \tilde{z}_2(\tau)\| 
  \mu(\tau)^{-\textnormal{sgn}(\tau)(\xi + \varepsilon)} \, d\tau \\
&\leq D \lambda \tilde{K} \|\eta_1 - \eta_2\|_{\infty,\mu} 
    \, \mu(t)^{\beta - a - \xi} 
\int_t^{+\infty} 
    \mu(\tau)^{-(2\xi - a + \beta) - 1} \mu'(\tau) \, d\tau \\
&= D \lambda \tilde{K} \|\eta_1 - \eta_2\|_{\infty,\mu} 
    \cdot \frac{\mu(t)^{-3\xi}}{2\xi - a + \beta} \\
&\leq \frac{D \lambda \tilde{K}}{2\xi - a + \beta} 
    \|\eta_1 - \eta_2\|_{\infty,\mu}.
\end{align*}
We now estimate the final term \(J_4(t)\). As in the previous cases, we proceed by bounding the corresponding integral for \(t \geq 0\). 
\begin{align*}
&\mu(t)^{-\textnormal{sgn}(t)(\xi+\varepsilon)} 
\int_t^{+\infty} \|\overline{T}(t,\tau) Q_0(\tau)\| 
\left| D_2g(\tau, T(\tau,t)b + \tilde{z}_1(\tau))\, w_1(\tau)
     - D_2g(\tau, T(\tau,t)b + \tilde{z}_2(\tau))\, w_2(\tau) \right| 
\|T(\tau,t)\| \, d\tau \\
&\leq \mu(t)^{-\textnormal{sgn}(t)(\xi+\varepsilon)} 
\int_t^{+\infty} D \left( \frac{\mu(t)}{\mu(\tau)} \right)^{\beta} 
\Big[ |D_2g(\tau, T(\tau,t)b + \tilde{z}_1(\tau))| 
      \left\| w_1(\tau) - w_2(\tau) \right\| \\
&\hspace{6em}
+ |D_2g(\tau, T(\tau,t)b + \tilde{z}_1(\tau)) 
    - D_2g(\tau, T(\tau,t)b + \tilde{z}_2(\tau))| 
    \left\| w_2(\tau) \right\| 
\Big] \\
&\quad \cdot \mu(\tau)^{\textnormal{sgn}(\tau)(\xi+\varepsilon)} 
\tilde{K} \left( \frac{\mu(\tau)}{\mu(t)} \right)^a 
\mu(t)^{\textnormal{sgn}(t)\varepsilon} \, d\tau \\
&\leq D \lambda \tilde{K} \, \mu(t)^{\beta - a - \textnormal{sgn}(t)\xi} 
\int_t^{+\infty} \mu(\tau)^{a - \beta - \textnormal{sgn}(\tau)(\xi - \varepsilon) - 1} 
\mu'(\tau) \left[ 
  \left\| w_1 - w_2 \right\|_{\infty,\mu} 
  + \|\eta_1 - \eta_2\|_{\infty,\mu} \cdot q 
\right] d\tau \\
&= D \lambda \tilde{K} 
\left[ \left\| w_1 - w_2 \right\|_{\infty,\mu} 
       + \|\eta_1 - \eta_2\|_{\infty,\mu} \cdot q 
\right] 
\mu(t)^{\beta - a - \xi} 
\int_t^{+\infty} \mu(\tau)^{-(\xi - a + \beta - \varepsilon) - 1} 
\mu'(\tau) \, d\tau \\
&= \frac{D \lambda \tilde{K}}{\xi - a + \beta - \varepsilon} 
\left[ \left\| w_1 - w_2 \right\|_{\infty,\mu} 
       + \|\eta_1 - \eta_2\|_{\infty,\mu} \cdot q 
\right] \mu(t)^{-2\xi + \varepsilon} \\
&\leq \frac{D \lambda \tilde{K}}{\xi - a + \beta - \varepsilon} 
\left\| \frac{\partial \eta_1}{\partial b} 
      - \frac{\partial \eta_2}{\partial b} \right\|_{\infty,\mu}
+ \frac{D \lambda \tilde{K} q}{\xi - a + \beta - \varepsilon} 
\|\eta_1 - \eta_2\|_{\infty,\mu}.
\end{align*}

\noindent
Collecting the bounds obtained for \( J_1, J_2, J_3 \) and \( J_4 \), we conclude that for \( t \geq 0 \),
\begin{align}
\left\|\frac{\partial F(\eta_1)}{\partial b} - \frac{\partial F(\eta_2)}{\partial b} \right\|_{\infty,\mu} 
<\;& \frac{D^2\lambda}{\alpha+\beta-2\xi} \|\eta_1 - \eta_2\|_{\infty,\mu} 
+ \frac{D^2\lambda}{\alpha+\beta-\xi+\varepsilon} 
  \left\| \frac{\partial \eta_1}{\partial b} - \frac{\partial \eta_2}{\partial b} \right\|_{\infty,\mu} \nonumber\\
&+ \frac{D^2\lambda q}{\alpha+\beta-\xi+\varepsilon} \|\eta_1 - \eta_2\|_{\infty,\mu} 
+ \frac{D\lambda\tilde{K}}{2\xi - a + \beta} \|\eta_1 - \eta_2\|_{\infty,\mu} \nonumber\\
&+ \frac{D\lambda\tilde{K}}{\xi - a + \beta - \varepsilon} 
  \left\| \frac{\partial \eta_1}{\partial b} - \frac{\partial \eta_2}{\partial b} \right\|_{\infty,\mu} 
+ \frac{D\lambda\tilde{K} q}{\xi - a + \beta - \varepsilon} \|\eta_1 - \eta_2\|_{\infty,\mu} \nonumber\\[0.5em]
<\;& \left( \frac{D^2\lambda}{\alpha+\beta-2\xi} + \frac{D\lambda\tilde{K}}{\xi - a + \beta - \varepsilon} \right) (1+q)\, \|\eta_1 - \eta_2\|_{\infty,\mu} \nonumber\\
&+ \left( \frac{D^2\lambda}{\alpha+\beta-2\xi} + \frac{D\lambda\tilde{K}}{\xi - a + \beta - \varepsilon} \right) 
  \left\| \frac{\partial \eta_1}{\partial b} - \frac{\partial \eta_2}{\partial b} \right\|_{\infty,\mu}. \label{Dfn12-tmayorcero}
\end{align}

Now, considering $t<0$, from \eqref{BG}, \eqref{B1}, \eqref{B2} and \eqref{Dg-cond}, and taking into account that $D_2g(t,0)=0$ and $b \in U(t)$, we obtain the following estimates. For clarity, we analyze each of the four terms individually, in a more concise manner than in the case $t \geq 0$.

\medskip
\noindent
First term. We estimate the integral involving the difference $D_2g(\cdot, \cdot + \eta_1) - D_2g(\cdot, \cdot + \eta_2)$:

\begin{equation*}
\begin{split}
\mu(t)^{-\textnormal{sgn}(t)(\xi+\varepsilon)} \int_{-\infty}^{t} & \|T_0(t,\tau)P_0(\tau)\|
\left| D_2g(\tau,\overline{T}(\tau,t)b+\eta_1(\tau,\overline{T}(\tau,t)b)) 
- D_2g(\tau,\overline{T}(\tau,t)b+\eta_2(\tau,\overline{T}(\tau,t)b)) \right|
\\
& \cdot \|\overline{T}(\tau,t)\| d\tau
\\
\leq \; & \frac{D^2 \lambda}{\alpha+\beta+2\xi} \|\eta_1 - \eta_2\|_{\infty,\mu}.
\end{split}
\end{equation*}

\medskip
\noindent
Second term. Now, we estimate the product of $D_2g(\cdot)$ with the difference of the derivatives $\partial\eta_1/\partial b - \partial\eta_2/\partial b$:

\begin{equation*}
\begin{split}
\mu(t)^{-\textnormal{sgn}(t)(\xi+\varepsilon)} \int_{-\infty}^{t} & \|T_0(t,\tau)P_0(\tau)\|
\Big| D_2g(\tau, \cdot + \eta_1) \frac{\partial\eta_1}{\partial b} - D_2g(\tau, \cdot + \eta_2) \frac{\partial\eta_2}{\partial b} \Big|
\|\overline{T}(\tau,t)\| \, d\tau
\\
\leq & \frac{D^2\lambda}{\alpha+\beta+\xi-\varepsilon} \left\| \frac{\partial\eta_1}{\partial b} - \frac{\partial\eta_2}{\partial b} \right\|_{\infty,\mu} + \frac{D^2\lambda q}{\alpha+\beta+\xi-\varepsilon} \|\eta_1 - \eta_2\|_{\infty,\mu}.
\end{split}
\end{equation*}

\medskip
\noindent
Third term. We now estimate the difference of $D_2g$ terms in the unstable integral (with $Q_0$):

\begin{equation*}
\begin{split}
\mu(t)^{-\textnormal{sgn}(t)(\xi+\varepsilon)} \int_{t}^{+\infty} & \|\overline{T}(t,\tau)Q_0(\tau)\|
\left| D_2g(\tau, T(\tau,t)b + \eta_1(\tau, T(\tau,t)b)) - D_2g(\tau, T(\tau,t)b + \eta_2(\tau, T(\tau,t)b)) \right| 
\\
& \cdot \|T(\tau,t)\| \, d\tau
\\
< & \frac{4\xi D\lambda \tilde{K}}{4\xi^2 - (a-\beta)^2} \|\eta_1 - \eta_2\|_{\infty,\mu}.
\end{split}
\end{equation*}

\medskip
\noindent
Fourth term. Finally, we estimate the cross term involving both the derivative and nonlinear variation:

\begin{equation*}
\begin{split}
\mu(t)^{-\textnormal{sgn}(t)(\xi+\varepsilon)} \int_{t}^{+\infty} & \|\overline{T}(t,\tau)Q_0(\tau)\|
\Big| D_2g(\cdot + \eta_1) \frac{\partial\eta_1}{\partial b} - D_2g(\cdot + \eta_2) \frac{\partial\eta_2}{\partial b} \Big|
\|T(\tau,t)\| \, d\tau
\\
< & \frac{2(\xi-\varepsilon)D\lambda\tilde{K}}{(\xi-\varepsilon)^2 - (a-\beta)^2} \left\| \frac{\partial\eta_1}{\partial b} - \frac{\partial\eta_2}{\partial b} \right\|_{\infty,\mu} \\
&+ \frac{2(\xi-\varepsilon)D\lambda\tilde{K}q}{(\xi-\varepsilon)^2 - (a-\beta)^2} \|\eta_1 - \eta_2\|_{\infty,\mu}.
\end{split}
\end{equation*}
Combining the estimates above, we obtain that for $t < 0$,
\begin{align}
        \Big\|\frac{\partial F(\eta_1)}{\partial b}-\frac{\partial F(\eta_2)}{\partial b}\Big\|_{\infty,\mu} <& \frac{D^2\lambda}{\alpha+\beta+2\xi}\|\eta_1-\eta_2\|_{\infty,\mu}+\frac{D^2\lambda}{\alpha+\beta+\xi-\varepsilon}\Big\|\frac{\partial\eta_1}{\partial b}-\frac{\partial\eta_2}{\partial b}\Big\|_{\infty,\mu}\nonumber\\
        &+\frac{D^2\lambda q}{\alpha+\beta+\xi-\varepsilon}\|\eta_1-\eta_2\|_{\infty,\mu} + \frac{4\xi D\lambda\tilde{K}}{4\xi^2-(a-\beta)^2}\|\eta_1-\eta_2\|_{\infty,\mu}\nonumber\\
        &+ \frac{2(\xi-\varepsilon)D\lambda\tilde{K}}{(\xi-\varepsilon)^2-(a-\beta)^2}\Big\|\frac{\partial\eta_1}{\partial b}-\frac{\partial\eta_2}{\partial b}\Big\|_{\infty,\mu} + \frac{2(\xi-\varepsilon)D\lambda\tilde{K}q}{(\xi-\varepsilon)^2-(a-\beta)^2}\|\eta_1-\eta_2\|_{\infty,\mu}\nonumber\\
        <& \frac{D^2\lambda}{\alpha+\beta+\xi-\varepsilon}(1+q)\|\eta_1-\eta_2\|_{\infty,\mu} + \frac{4\xi D\lambda\tilde{K}}{(\xi-\varepsilon)^2-(a-\beta)^2}(1+q)\|\eta_1-\eta_2\|_{\infty,\mu}\nonumber\\
        &+ \Big(\frac{D^2\lambda}{\alpha+\beta+\xi-\varepsilon}+\frac{4\xi D\lambda\tilde{K}}{(\xi-\varepsilon)^2-(a-\beta)^2}\Big)\Big\|\frac{\partial\eta_1}{\partial b} - \frac{\partial\eta_2}{\partial b}\Big\|_{\infty,\mu}\nonumber\\
        =& \Big(\frac{D^2\lambda}{\alpha+\beta+\xi-\varepsilon}+\frac{4\xi D\lambda\tilde{K}}{(\xi-\varepsilon)^2-(a-\beta)^2}\Big)(1+q)\|\eta_1-\eta_2\|_{\infty,\mu}\nonumber\\
        &+ \Big(\frac{D^2\lambda}{\alpha+\beta+\xi-\varepsilon}+\frac{4\xi D\lambda\tilde{K}}{(\xi-\varepsilon)^2-(a-\beta)^2}\Big)\Big\|\frac{\partial\eta_1}{\partial b} - \frac{\partial\eta_2}{\partial b}\Big\|_{\infty,\mu}.\label{Dfn12-tmenorcero}
\end{align}

Hence, for $t\in\mathbb{R}$, it follows from \eqref{Dfn12-tmayorcero} and \eqref{Dfn12-tmenorcero} that 

\begin{align}
    \Big\|\frac{\partial F(\eta_1)}{\partial b} - \frac{\partial F(\eta_2)}{\partial b}\Big\|_{\infty,\mu} <& \Big(\frac{D^2\lambda}{\alpha+\beta-2\xi}+\frac{4\xi D\lambda\tilde{K}}{(\xi-\varepsilon)^2-(a-\beta)^2}\Big)(1+q)\|\eta_1-\eta_2\|_{\infty,\mu}\nonumber\\
    & + \Big(\frac{D^2\lambda}{\alpha+\beta-2\xi}+\frac{4\xi D\lambda\tilde{K}}{(\xi-\varepsilon)^2-(a-\beta)^2}\Big)\Big\|\frac{\partial \eta_1}{\partial b} - \frac{\partial\eta_2}{\partial b}\Big\|_{\infty,\mu}.\label{Dfn12-cota}
\end{align}

From \eqref{Fcont} and \eqref{Dfn12-cota} we obtain

\begin{align*}
    \|F(\eta_1)-F(\eta_2)\|_{1,\mu} = & \|F(\eta_1)-F(\eta_2)\|_{\infty,\mu} + \Big\|\frac{\partial F(\eta_1)}{\partial b} - \frac{\partial F(\eta_2)}{\partial b}\Big\|_{\infty,\mu}\nonumber\\
    <& D\delta\Big(\frac{\alpha+\beta}{\alpha\beta}\Big)\|\eta_1-\eta_2\|_{\infty,\mu}\nonumber\\
    &+ \Big(\frac{D^2\lambda}{\alpha+\beta-2\xi}+\frac{4\xi D\lambda\tilde{K}}{(\xi-\varepsilon)^2-(a-\beta)^2}\Big)(1+q)\|\eta_1-\eta_2\|_{\infty,\mu}\nonumber\\
    &+  \Big(\frac{D^2\lambda}{\alpha+\beta-2\xi}+\frac{4\xi D\lambda\tilde{K}}{(\xi-\varepsilon)^2-(a-\beta)^2}\Big)\Big\|\frac{\partial \eta_1}{\partial b} - \frac{\partial\eta_2}{\partial b}\Big\|_{\infty,\mu}\nonumber\\
    <& \Big[D\delta\Big(\frac{\alpha+\beta}{\alpha\beta}\Big)\nonumber\\
    &+ \frac{D\lambda[D((\xi-\varepsilon)^2-(a-\beta)^2)+4\xi\tilde{K}(\alpha+\beta-2\xi)]}{(\alpha+\beta-2\xi)[(\xi-\varepsilon)^2-(a-\beta)^2]}(1+q)\Big]\|\eta_1-\eta_2\|_{1,\mu}.
\end{align*}

The conditions \eqref{ineq-delta}, \eqref{ineq-lambda} and the above inequality imply that
\begin{align*}
   \|F(\eta_1)-F(\eta_2)\|_{1,\mu} < & \Big[\frac{q}{(1+q)^2}+\frac{q^2}{(1+q)^2}\Big]\|\eta_1-\eta_2\|_{1,\mu}\\
   =& \frac{q}{1+q}\|\eta_1-\eta_2\|_{1,\mu},
\end{align*}
 and since $\dfrac{q}{1+q}<1$, it is proven that $F$ is a contraction.
\end{proof}

We are now in position to state and prove the main result of this work. The proof relies on the estimates established in the previous lemmas, which guarantee both the existence of a fixed point for the operator $F$ and the smoothness properties of the resulting conjugacy.

\begin{theorem}
\label{Th-Diff}
   Let \(q > 0\) and assume that hypotheses \textnormal{(H1)}, \textnormal{(H3)}, and \textnormal{(H4)} hold. In addition, suppose that the parameters \(\delta\) and \(\lambda\) comply with conditions \eqref{ineq-delta} and \eqref{ineq-lambda}, respectively.
Then there exists a unique $\eta\in\mathcal{M}$ such that $h^{t}=Id_{U(t)}+\eta^t$ and 
    \begin{equation}\label{conj-diff}
        h^t\circ T(t,s) = R(t,s)\circ h^s \quad \textnormal{ on } U(s),
    \end{equation}
    for every $t, s\in\mathbb{R}$ with $t\geq s$. Moreover, each map $h^t$ is a $C^1$-- diffeomorphism.
\end{theorem}

\begin{proof}
By Lemma~\ref{lemma:contraction}, the operator \( F \) is a contraction on \( \mathcal{B} \). Hence, by the Banach Fixed Point Theorem, there exists a unique \( \eta \in \mathcal{B} \) such that \( F(\eta) = \eta \). It was shown in \cite[Lemma 2]{B-v-per} that the property \eqref{conj-diff} holds if and only if \( F(\eta) = \eta \). Although our operator acts on different function spaces, the argument can be adapted without essential changes, and the conclusion remains valid in our setting.

Furthermore, from the definition of \( h \), we have
\[
\frac{\partial h(t,b)}{\partial b} = \operatorname{Id}_{U(t)} + \frac{\partial \eta(t,b)}{\partial b}.
\]
As a consequence of Lemmas~\ref{well defined} and~\ref{lemma:contraction}, we obtain the bound
\[
\left\| \frac{\partial \eta}{\partial b} \right\|_{\infty,\mu}
\leq \| \eta \|_{1,\mu}
= \| F(\eta) \|_{1,\mu}
\leq \frac{q}{1+q} < 1,
\]
which implies that \( \frac{\partial h(t,b)}{\partial b} \) is invertible for all \( t \in \mathbb{R} \). Therefore, each map \( h^t \colon U(t) \to \mathbb{R}^n \) is a diffeomorphism. We conclude that \( h^t \) is of class \( C^1 \) for every \( t \in \mathbb{R} \).

\end{proof}

To illustrate the applicability of the theorem, we present a concrete example of parameters that satisfy all the hypotheses and ensure the validity of the result. This example helps to clarify how the abstract conditions translate into explicit numerical bounds and serves as a guide for constructing admissible systems.

\begin{example}
Consider the parameters $\alpha=0.8$, $\beta=0.6$, $\theta=0.4$, $\nu=0.2$, $a=1$ and $\varepsilon=0.1$. It can be easily verified that the parameters $\alpha$, $\beta$, $\theta$, $\nu$ and $\varepsilon$ fulfill \eqref{hyp:mu-dichotomy}. Moreover, suppose that there exist \(\delta > 0\), \(\gamma > 0.4\), \(\lambda > 0\) and \(0.5<\xi<0.7\) satisfying  \eqref{g-cond2} and \eqref{Dg-cond}. Let $q>0$, then if $\delta\leq\dfrac{0.34q}{D(1+q)^2}$ and $\lambda\leq\dfrac{0.018q^2}{D(0,09D+0.48\tilde{K})(1+q)^3}$ we obtain the inequalities  \eqref{ineq-delta} and \eqref{ineq-lambda}. Hence, the hypotheses of Theorem \ref{Th-Diff} are satisfied.
\end{example}

\section*{Declaration of competing interest}

The authors declare that they have no known competing financial
interests or personal internships that could have appeared to influence the work reported
on the paper.

\section*{Declaration of availability of data}

The manuscript has not associated data.

\end{document}